\pdfoutput=1
\documentclass[letterpaper]{amsart}

\usepackage[utf8]{inputenc}
\usepackage{lmodern}
\usepackage[T1]{fontenc}
\usepackage{amssymb}
\usepackage{enumitem}
\usepackage{mathtools}

\usepackage{tikz-cd}
\usepackage[pdfusetitle]{hyperref}
\usepackage{cleveref}

\tikzcdset{arrow style=Latin Modern}

\newlist{enumarabic}{enumerate}{1}
\setlist[enumarabic]{font=\normalfont,label=(\arabic*),leftmargin=0.3in}
\newlist{enumroman}{enumerate}{1}
\setlist[enumroman]{font=\normalfont,label=(\roman*),leftmargin=0.3in}

\numberwithin{equation}{section}
\allowdisplaybreaks[4]

\theoremstyle{plain}
\newtheorem{theorem}{Theorem}[section]
\newtheorem{proposition}[theorem]{Proposition}
\newtheorem{lemma}[theorem]{Lemma}
\newtheorem{corollary}[theorem]{Corollary}

\theoremstyle{definition}
\newtheorem{remark}[theorem]{Remark}

\theoremstyle{remark}
\newtheorem*{acknowledgements}{Acknowledgements}

\DeclareMathAlphabet\mathbfit{OML}{cmm}{b}{it}

\def\citeorbitsone#1{\cite[#1]{AlldayFranzPuppe:orbits1}}
\def\citeorbitsfour#1{\cite[#1]{AlldayFranzPuppe:orbits4}}

\let\newterm\emph

\def\arxiv#1{\href{http://arxiv.org/abs/#1}{\texttt{arXiv:#1}}}
\def\doi#1{\href{http://doi.org/#1}{doi:#1}}

\def\cf{\emph{cf.}}

\let\epsilon\varepsilon
\let\phi\varphi
\let\emptyset\varnothing

\def\N{\mathbb N}
\def\Z{\mathbb Z}
\def\Q{\mathbb Q}
\def\FF{\mathbb F}

\DeclareMathOperator{\Hom}{Hom}
\DeclareMathOperator{\Ext}{Ext}
\DeclareMathOperator{\coker}{coker}
\DeclareMathOperator{\depth}{depth}
\DeclareMathOperator{\height}{ht}
\DeclareMathOperator{\projdim}{proj\,dim}
\DeclareMathOperator{\im}{im}

\DeclareMathOperator{\syzord}{syzord}

\def\pair#1{{\langle#1\rangle}}

\def\kk{\Bbbk}
\def\kktilde{\skew{-2}\tilde\Bbbk}

\def\pp{\mathfrak{p}}
\def\qq{\mathfrak{q}}

\def\UU{\mathcal U}

\def\hatX{\widehat X}

\def\AB{AB}
\def\barAB{{\smash{\overline{AB\mathstrut}}\mathstrut}}

\def\ABT{\AB_{T}}
\def\ABG{\AB_{G}}
\def\barABT{\barAB_{T}}
\def\barABG{\barAB_{G}}

\def\Hc{H_{c}}
\def\HT{H_{T}}
\def\HTc{H_{T,c}}
\def\CT{C_{T}}

\def\hHc{H^{c}}
\def\hHT{H^{T\!}}
\def\hHTc{H^{T,c}}

\def\HG{H_{G}}
\def\HGc{H_{G,c}}
\def\hHG{H^{G}}
\def\hHGc{H^{G,c}}

\def\HK{H_{K}}
\def\CK{C_{K}}
\def\HL{H_{L}}

\def\AA{;\ell}

\def\RT{R}
\def\RG{R_{G}}

\DeclareMathOperator{\rank}{rank}
\def\F{\mathbb{F}}
\def\Hhc{H_{hc}}
\def\hHhc{H^{hc}}

\def\indT#1{\tilde #1}

\def\dirlim{\mathop{\underrightarrow\lim}}

\def\DM{\mathbfit{D}}

\def\Xp#1{X_{#1}}
\def\hatXp#1{\hatX_{#1}}
\def\Xpp#1{X'_{#1}}

\def\mm{\mathfrak{m}}

\def\HH{\mathcal{H}}
\def\QQ{\mathcal{Q}}

\def\Equiv#1#2{$\hbox{\ref{#1}}\Leftrightarrow\hbox{\ref{#2}}$}
\def\impl#1#2{$\hbox{\ref{#1}}\Rightarrow\hbox{\ref{#2}}$}
\def\implthree#1#2#3{$\hbox{\ref{#1}}\Rightarrow\hbox{\ref{#2}}\Rightarrow\hbox{\ref{#3}}$}
\def\refstar#1{\ref{#1}$^{*}$}

\def\spanofel#1{\langle#1\rangle}

\begin{document}

\title[Syzygies in equivariant cohomology]{Syzygies in equivariant cohomology\\in positive characteristic}
\author{Christopher Allday}
\address{Department of Mathematics, University of Hawaii,
  2565~McCarthy Mall, Honolulu, HI~96822, U.S.A.}
\email{chris@math.hawaii.edu}
\author{Matthias Franz}
\address{Department of Mathematics, University of Western Ontario,
      London, Ont.\ N6A\;5B7, Canada}
\email{mfranz@uwo.ca}
\author{Volker Puppe}
\address{Fachbereich Mathematik, Universität Konstanz, 78457 Konstanz, Germany}
\email{volker.puppe@uni-konstanz.de}
\thanks{M.\,F.\ was supported by an NSERC Discovery Grant.}

\pdfstringdefDisableCommands{\def\and{, }}
\hypersetup{pdfauthor=\authors,pdftitle={Syzygies in equivariant cohomology in positive characteristic}}

\subjclass[2020]{Primary 55N91; secondary 13D02, 55M35, 57P10}

\begin{abstract}
  We develop a theory of syzygies in equivariant cohomology
  for tori as well as \(p\)-tori and coefficients in~\(\FF_{p}\).
  A noteworthy feature is a new algebraic approach to the partial exactness
  of the Atiyah--Bredon sequence, which also covers all instances considered so far.
\end{abstract}

\maketitle

\section{Introduction}

In the papers~\cite{AlldayFranzPuppe:orbits1},~\cite{AlldayFranzPuppe:orbits4}
we studied the Borel-equivariant cohomology of spaces equipped with an action of a torus~\(T=(S^{1})^{r}\)
and in particular the relation between the so-called Atiyah--Bredon sequence (see~\eqref{eq:barAB} below)
and the notion of syzygies coming from commutative algebra.
This was generalized to arbitrary compact connected Lie groups in~\cite{Franz:nonab}.
Coefficients were taken in a field of characteristic~\(0\).

In the present paper we develop an analogous theory
for Borel-equivariant cohomology with coefficients in a field~\(\kk\) of characteristic~\(p>0\),
for actions of tori as well as of \(p\)-tori~\(G=(\Z_{p})^{r}\subset T\).
For example, under mild hypotheses on the \(G\)-space~\(X\) stated in \Cref{sec:torus-equiv},
we characterize the exactness of the \newterm{Chang--Skjelbred sequence}
\begin{equation}
  \label{eq:intro:cs}
  \let\longrightarrow\rightarrow
  0
  \longrightarrow \HG^{*}(X,\kk)
  \longrightarrow \HG^{*}(\Xp{0};\kk)
  \longrightarrow \HG^{*+1}(\Xp{1}, \Xp{0};\kk)
\end{equation}
where \(\Xp{0}=X^{G}\) denotes the fixed point set
and \(\Xp{1}\subset X\) the union of all orbits with at most \(p\)~elements.
Recall that the polynomial ring~\(\RT=H^{*}(BT;\kk)\) injects into~\(H^{*}(BG;\kk)\),
so that \(\HG^{*}(X)\) naturally becomes an \(\RT\)-module. Chang--Skjelbred~\cite[Lemma~2.3]{ChangSkjelbred:1974} proved
that \eqref{eq:intro:cs} is exact if \(\HG^{*}(X)\) is free over~\(\RT\).
This gives a powerful tool to compute \(\HG^{*}(X)\) out of the equivariant \(1\)-skeleton~\(\Xp{1}\),
nowadays often called the ``GKM~method'' after Goresky--Kottwitz--MacPherson~\cite[Thm.~7.2]{GoreskyKottwitzMacPherson:1998}.
A version of the Chang--Skjelbred lemma for free torus-equivariant cohomology with
coefficients in~\(\kk\) has appeared in~\cite[Thm.~2.1]{FranzPuppe:2011}.

A special case of our main result for \(p\)-tori (\Cref{thm:G:partial}) is the following.
Remember that an \(\RT\)-module is reflexive if the canonical map to its double-dual is an isomorphism.
The reflexive \(\RT\)-modules are exactly the second syzygies.
Finitely generated free modules are reflexive, but there are many others, see \Cref{sec:examples} for examples in equivariant cohomology.

\begin{theorem}
  \label{thm:intro:cs-reflexive}
  The Chang--Skjelbred sequence~\eqref{eq:intro:cs} is exact
  if and only if \(\HG^{*}(X;\kk)\) is a reflexive \(\RT\)-module.
\end{theorem}

In fact, all major results of~\cite{AlldayFranzPuppe:orbits1} and~\cite{AlldayFranzPuppe:orbits4}
continue to hold in this new setting. Nevertheless, some of them require new methods of proof,
ultimately because \(G\) has only finitely many subgroups and the field~\(\FF_{p}\) only finitely many elements.

Another contribution of the present paper therefore is the development of new algebraic techniques
to prove \Cref{thm:intro:cs-reflexive} and its generalization to arbitrary syzygies.
It also applies to the cases considered previously (tori and compact connected Lie groups, coefficients in a field of characteristic~\(0\))
and thus unifies the whole theory.
In particular, we isolate the properties an equivariant cohomology theory must satisfy
in order for our methods to be applicable, see \Cref{rem:generality}. We anticipate that this opens the way to further generalizations.

The paper is organized as follows:
After a review of some commutative algebra in \Cref{sec:alg-prelim},
we prove in \Cref{sec:alg-partial} the algebraic result on which our characterization
of the partial exactness of the Atiyah--Bredon sequence is based.
Then we turn to topology. In \Cref{sec:top-prelim} we state
a converse to the Leray--Hirsch theorem as well as
our standing assumptions and review torus-equivariant (co)homology.
In \Cref{sec:ind-red} we investigate how equivariant cohomology behaves under induction and restriction of the action.
The crucial Cohen--Macaulay property of the orbit filtration is established in \Cref{sec:orbit-cm}.
Equivariant homology for \(G\)-spaces is defined in \Cref{sec:equiv-hom}. Our main results
appear in \Cref{sec:main} and examples in \Cref{sec:examples}.
We conclude with remarks about the cases~\(p=2\) and odd~\(p\) separately in \Cref{sec:remarks-2-odd}.
In particular, we discuss there that for actions of \(2\)-tori one can work with the polynomial ring~\(H^{*}(BG;\kk)\)
instead of~\(\RT=H^{*}(BT)\) as far as syzygies are concerned. The focus of the present paper however
is to develop methods that apply equally to the case of odd~\(p\).
Throughout, we stress new methods and refer the reader to our previous papers
instead of repeating proofs.

\begin{acknowledgements}
  M.\,F.\ thanks Winfried Bruns for helpful discussions
  and Sergio Chaves for bringing Dwyer's work~\cite{Dwyer:1974} to his attention.
\end{acknowledgements}

\section{Algebraic preliminaries}
\label{sec:alg-prelim}

\subsection{Depth, Cohen--Macaulay modules and syzygies}

We always assume Noetherian rings to be commutative.

Let \(S\) be a regular local ring of dimension~\(d<\infty\), and let \(N\) be a finitely generated \(S\)-module.
Recall that the depth of~\(N\) is the maximal length of an \(N\)-regular sequence in the maximal ideal~\(\mm\lhd S\).
The depth of the zero module is \(\infty\). For non-zero~\(N\) one has the bounds
\begin{equation}
  \label{eq:bound-depth}
  \depth N\le \dim N \le \dim S < \infty,
\end{equation}
compare~\cite[Folgerung~6.10, Satz~14.18\,(b)]{BrueskeIschebeckVogel:1989}.
If the equality
\begin{equation}
  \label{eq:def-CM-local}
  \depth N = \dim N
\end{equation}
holds or if \(N=0\), then \(N\) is \newterm{Cohen--Macaulay}.

If \(N\) is non-zero, then its projective dimension is finite, \cf~\cite[Thm.~14.22]{BrueskeIschebeckVogel:1989}.
It is related to the depth via the Auslander--Buchsbaum formula
\begin{equation}
  \label{eq:auslander-buchsbaum}
  \projdim N + \depth N = \depth S,
\end{equation}
compare~\cite[Satz~14.19]{BrueskeIschebeckVogel:1989}.
The characterization of projective dimension via~\(\Ext\),
\begin{align}
  \label{eq:projdim-Ext}
  \projdim N &= \max \bigl\{\, i \bigm| \Ext^{i}_{S}(N,S)\ne 0 \,\bigr\},
\shortintertext{leads to the formula}
  \label{eq:depth-Ext}
  \depth N &= \min \bigl\{\, i \bigm| \Ext^{d-i}_{S}(N,S)\ne 0 \,\bigr\},
\end{align}
see~\cite[eq.~(2.4)]{Franz:orbits3}.

Now let \(R\) be a finite-dimensional regular ring and \(M\) a finitely generated \(R\)-module.
For a prime ideal~\(\pp\lhd R\), the dimension of the localized ring~\(R_{\pp}\)
is the height of~\(\pp\),
\begin{equation}
  \label{eq:local-dim-height}
  \dim R_{\pp} = \height\pp,
\end{equation}
\cf~\cite[Bemerkung~6.2\,(a)]{BrueskeIschebeckVogel:1989}.
If \(R\) is a polynomial ring over a field and if \(\pp\) is generated by homogeneous linear polynomials,
then \(\height\pp\) equals the dimension of the vector space spanned by these generators.

We will also need the following standard fact.

\begin{lemma}
  \label{thm:local-depth-0}
  If \(M\ne0\), then there is a prime ideal~\(\pp\lhd R\) such that \(\depth M_{\pp}=0\).
\end{lemma}

\begin{proof}
  Take a~\(\pp\) associated to~\(M\). Then \(\pp_{\pp}\lhd R_{\pp}\) is associated to~\(M_{\pp}\),
  hence \(\depth M_{\pp}=0\), compare~\cite[Folgerung~4.5, Satz~4.11, Feststellung~14.18]{BrueskeIschebeckVogel:1989}.
\end{proof}

Recall that \(M\) is defined to be Cohen--Macaulay over~\(R\) if so is \(M_{\pp}\) over~\(R_{\pp}\) for every prime ideal~\(\pp\lhd R\),
\cf~\cite[Def.~17.19]{BrueskeIschebeckVogel:1989}. We say that \(M\) is
\newterm{zero or Cohen--Macaulay of projective dimension~\(i\ge0\)}
if every localization~\(M_{\pp}\) is so. Since
\begin{equation}
  \projdim M = \sup_{\pp}\,\projdim M_{\pp},
\end{equation}
compare~\cite[Folgerung~13.25]{BrueskeIschebeckVogel:1989}, this entails that \(M\) itself is of projective dimension~\(i\).

\begin{lemma}
  \label{thm:CM-Ext}
  The R-module~\(M\) is zero or Cohen--Macaulay of projective dimension~\(i\ge0\) if and only if
  \(\Ext_{R}^{j}(M,R)\) vanishes for~\(j\ne i\).
\end{lemma}

\begin{proof}
  The vanishing condition implies
  \begin{equation}
    \label{eq:CM-Ext-pp}
    \Ext_{R_{\pp}}^{j}(M_{\pp},R_{\pp}) \ne 0 \:\Longrightarrow\: j = i
  \end{equation}
  for any prime ideal~\(\pp\lhd R\) and any~\(j\),
  which by~\eqref{eq:projdim-Ext} shows that \(M_{\pp}\) is either zero or Cohen--Macaulay of projective dimension~\(i\) over~\(R_{\pp}\) in this case.
  
  Conversely, if \(M\) is non-zero and all~\(M_{\pp}\) are zero or Cohen--Macaulay of projective dimension~\(i\), then
  the vanishing condition must hold for otherwise some localization would violate \eqref{eq:CM-Ext-pp}.
\end{proof}

If \(R\) is a polynomial ring in \(r\)~variables over a field, then the vanishing condition in \Cref{thm:CM-Ext}
is for non-zero~\(M\) equivalent to the identities
\begin{equation}
  \depth M = \dim M = r-i
\end{equation}
analogous to~\eqref{eq:def-CM-local} and~\eqref{eq:auslander-buchsbaum}, see~\cite[Thm.~A1.9~\&~Prop.~A1.16]{Eisenbud:2005}.
Moreover, for non-zero~\(M\) the analogue of the Auslander--Buchsbaum formula holds,
\begin{equation}
  \projdim M + \depth M = r,
\end{equation}
compare~\cite[Rem.~A.6.17]{AlldayPuppe:1993}.

\begin{lemma}
  \label{thm:cm-ext}
  If \(M\) is zero or Cohen--Macaulay of projective dimension~\(i\ge0\),
  then so is \(N=\Ext_{R}^{i}(M,R)\). Moreover, we have \(M=\Ext_{R}^{i}(N,R)\) in this case.
\end{lemma}

\begin{proof}
  Recall that the dual of a finitely generated projective module is again (finitely generated and) projective,
  and also that any finitely generated projective module is reflexive.
  
  Our claim holds for~\(i=0\) because \(M\) and~\(N\) are projective in this case.
  So assume \(i>0\), which implies that \(M^{\vee}\coloneqq\Hom_{R}(M,R)\) vanishes. Let
  \begin{equation}
    \label{eq:cm-ext-1}
    0 \longrightarrow L_{i} \stackrel{f_{i}}{\longrightarrow} \cdots \stackrel{f_{1}}{\longrightarrow} L_{0} \stackrel{f_{0}}{\longrightarrow} M \longrightarrow 0
  \end{equation}
  be a finitely generated projective resolution of~\(M\). Taking \(R\)-duals and using \Cref{thm:CM-Ext}, we see that \(N=\coker(f_{i}^{\vee})\)
  has the finitely generated projective resolution
  \begin{equation}
    M^{\vee} = 0 \longrightarrow L_{0}^{\vee} \stackrel{f_{1}^{\vee}}{\longrightarrow} \cdots \stackrel{f_{i}^{\vee}}{\longrightarrow} L_{i}^{\vee} \longrightarrow N \longrightarrow 0.
  \end{equation}
  Taking duals again brings us back to the exact sequence~\eqref{eq:cm-ext-1} with~\(M=\coker f_{1}\) and completes the proof.
\end{proof}

We say that \(M\) is a \newterm{\(j\)-th syzygy} for some~\(j\ge0\) if there is an exact sequence
\begin{equation}
  \label{eq:def-syzygy}
  0 \to M \to F^{1} \to \dots \to F^{j}
\end{equation}
with finitely generated free \(R\)-modules~\(F^{1}\),~\dots,~\(F^{j}\).
The syzygy order of~\(M\), written \(\syzord M\), is the largest~\(j\) (possibly \(\infty\)) such that \(M\) is a \(j\)-th syzygy.

We will use the following characterization of syzygies from~\cite[Prop.~16.33]{BrunsVetter:1988}.

\begin{proposition}
  \label{thm:char-syzygy}
  The following are equivalent for any~\(j\ge0\).
  \begin{enumarabic}
  \item \(M\) is a \(j\)-th syzygy.
  \item Any \(R\)-regular sequence of length at most~\(j\) is \(M\)-regular.
  \item For every prime ideal~\(\pp\lhd R\) we have
    \begin{equation*}
      \depth M_{\pp} \ge \min(j,\depth R_{\pp}).
    \end{equation*}
  \end{enumarabic}
\end{proposition}

In particular, \(M\) is a first syzygy if and only if it is torsion-free.
Moreover, \(M\) is a second syzygy if and only if it is reflexive \cite[Props.~16.31\,(b)~\&~16.33]{BrunsVetter:1988}.
(This is also equivalent to \(M\) being the dual of a finitely generated \(R\)-module.)
For a polynomial ring in \(r\)~indeterminates over a field, Hilbert's syzygy theorem
implies that the \(r\)-th (or any higher) syzygies are free \(R\)-modules;
in the graded setting this still holds if one drops the requirement of finite generation in~\eqref{eq:def-syzygy}, see~\cite[Thm.~XXI.4.15]{Lang:1993}.

\begin{lemma}
  \label{thm:syz-prod-alg}
  Let \(R_{1}\) and~\(R_{2}\) be polynomial algebras over a field~\(\kk\),
  and let \(M_{1}\) and~\(M_{2}\) be finitely generated modules over~\(R_{1}\) and~\(R_{2}\), respectively. Then
  \begin{equation*}
    \syzord_{R_{1}\otimes_{\kk} R_{2}}(M_{1}\otimes_{\kk} M_{2}) = \min\bigl(\syzord_{R_{1}} M_{1},\syzord_{R_{2}} M_{2}\bigr).
  \end{equation*}
\end{lemma}

\begin{proof}
  If, say, \(s=\syzord M_{1}\) is finite, then there is an \(R_{1}\)-regular sequence of length~\(s+1\) that is not \(M_{1}\)-regular.
  The same sequence then is regular on~\(R_{1}\otimes R_{2}\), but not on~\(M_{1}\otimes M_{2}\). This proves the inequality~``\(\le\)'' in the claimed formula.

  Now let
  \begin{equation}
    \let\to\longrightarrow
    0 \to M_{1} \to F_{1}^{0} \to \dots \to F_{1}^{j_{1}-1}
  \end{equation}
  be exact, where \(F_{1}^{*}\) is a complex of finitely generated free \(R_{1}\)-module,
  and take an analogous complex~\(0 \to M_{2} \to F_{2}^{*}\), assuming \(j_{2}\ge j_{1}\ge 1\).
  
  The total complex~\(F_{1}^{*}\otimes F_{2}^{*}\) is one of finitely generated free modules over~\(R_{1}\otimes R_{2}\),
  and \(M_{1}\otimes M_{2}\) injects into~\(F_{1}^{0}\otimes F_{2}^{0}\).
  If \(j_{1}\ge 2\), we moreover have
  \begin{equation}
    H^{k}(F_{1}^{*}\otimes F_{2}^{*}) =
    \begin{cases}
      M_{1}\otimes M_{2} & \text{for \(k=0\),} \\
      0 & \text{for~\(1\le k< j_{1}\).}
    \end{cases}
  \end{equation}
  This exhibits \(M_{1}\otimes M_{2}\) as a \(j_{1}\)-th syzygy and completes the proof.
\end{proof}

\section{An algebraic version of partial exactness}
\label{sec:alg-partial}

The purpose of this section is to isolate the algebraic argument
that is used in \Cref{sec:partial} to characterize the partial exactness
of the Atiyah--Bredon sequence.

\begin{lemma}
  \label{thm:seq-Lk}
  Let \(S\) be a Noetherian local ring, and let \(k\ge0\). Assume that
  \begin{equation*}
    0 \to M \to L_{k} \to L_{k-1} \to \dots \to L_{1} \to L_{0} \to 0
  \end{equation*}
  is an exact sequence of finitely generated \(S\)-modules
  such that \(\depth L_{i}\ge i\) for all~\(i\).
  Then \(\depth M\ge k\). Moreover, if \(\depth M > k\), then \(\depth L_{0}>0\).
\end{lemma}

\begin{proof}
  Consider the short exact sequence
  \begin{equation}
    0 \longrightarrow M \longrightarrow L_{k} \longrightarrow N \longrightarrow 0.
  \end{equation}
  Then
  \begin{align}
    \depth N \ge k-1 &\;\;\Rightarrow\;\;
    \depth M \ge \min\bigl(\depth L_{k},\depth N+1) \ge k \\
    \shortintertext{and}
    \depth M > k &\;\;\Rightarrow\;\;
    \depth N \ge \min\bigl(\depth L_{k},\depth M-1) > k-1,
  \end{align}
  \cf~\cite[Prop.~16.14]{BrunsVetter:1988}.
  Our claims follow from this by induction.
\end{proof}

Let \(R\) be a regular ring of dimension~\(d<\infty\), and let \(M\) be a finitely generated \(R\)-module.
We consider a complex~\(K^{*}\) of finitely generated \(R\)-modules,
\begin{equation}
  K^{0} \stackrel{\delta_{0}}{\longrightarrow} K^{1}  \stackrel{\delta_{1}}{\longrightarrow} \cdots
  \stackrel{\delta_{d-1}}{\longrightarrow} K^{d} \longrightarrow 0,
\end{equation}
together with a map~\(\iota\colon M\to K^{0}\) such that \(\delta_{0}\,\iota=0\).
We can then form the augmented complex~\(\bar K^{*}\) with~\(\bar K^{-1}=M\) and~\(\delta_{-1}=\iota\),
\begin{equation}
  0 \longrightarrow M \stackrel{\iota}{\longrightarrow} K^{0} \stackrel{\delta_{0}}{\longrightarrow} K^{1}  \stackrel{\delta_{1}}{\longrightarrow} \cdots
  \stackrel{\delta_{d-1}}{\longrightarrow} K^{d} \longrightarrow 0.
\end{equation}
We make the following two assumptions:
\begin{enumarabic}
\item
  \label{ass:CM}
  For every~\(i\ge0\), the \(R\)-module~\(K^{i}\) is zero or Cohen--Macaulay of projective dimension~\(i\).
\item
  \label{ass:exact}
  If the localization~\(\bar K^{*}_{\pp}\) at a prime ideal~\(\pp\lhd R\) is exact at all but possibly two adjacent positions, then it is exact everywhere.
\end{enumarabic}

Using the Auslander--Buchsbaum formula~\eqref{eq:auslander-buchsbaum}
(as well as \(\depth 0=\infty\)), we deduce from assumption~\ref{ass:CM} that
for any~\(0\le i\le d\) and any prime ideal~\(\pp\lhd R\)
we have
\begin{gather}
  \label{eq:ineq-depth}
  \depth K^{i}_{\pp} \in \bigl\{ \depth R_{\pp}-i, \infty \bigr\} \\
  \shortintertext{and in particular}
  \label{eq:local-0}
  i > \depth R_{\pp}
  \quad\Longrightarrow\quad
  K^{i}_{\pp}=0.
\end{gather}
(Recall from~\eqref{eq:bound-depth} that the depth of a non-zero \(R_{\pp}\)-module is finite.)

\begin{theorem}
  \label{thm:partial-alg}
  Let \(K^{*}\) and~\(M\) be as above, and let \(j\ge0\). Then
  \(M\) is a \(j\)-th syzygy over~\(R\) if and only if
  \(H^{i}(\bar K^{*})=0\) for all~\(-1\le i\le j-2\).
\end{theorem}

\begin{proof}
  We assume \(j\ge1\) as the claim is void for~\(j=0\).
  
  \(\Leftarrow\):
  Let \(\pp\lhd R\) be prime and write \(d=\depth R_{\pp}\).
  If \(d<j\), then \(d-1\le j-2\), and the sequence
  \begin{equation}
    0\longrightarrow M_{\pp}\longrightarrow K^{0}_{\pp}
    \longrightarrow\cdots
    \longrightarrow K^{d-1}_{\pp}
    \longrightarrow (\im\delta_{d-1})_{\pp}
    \longrightarrow 0
  \end{equation}
  is exact. \Cref{thm:seq-Lk} together with~\eqref{eq:ineq-depth} gives \(\depth M_{\pp}\ge d=\min(j,d)\).
  
  For~\(d\ge j\) the sequence
  \begin{equation}
    0\longrightarrow M_{\pp}\longrightarrow K^{0}_{\pp}
    \longrightarrow\cdots
    \longrightarrow K^{j-2}_{\pp}
    \longrightarrow (\im\delta_{j-2})_{\pp}
    \longrightarrow 0
  \end{equation}
  is exact, and
  \begin{equation}
    \label{eq:depth-ge-ji}
    \depth K^{i}_{\pp}\ge d-i\ge j-i
  \end{equation}
  for all~\(i\ge0\).
  Since \((\im\delta_{j-2})_{\pp}\) is a submodule of~\(K^{j-1}_{\pp}\),
  we additionally have
  \begin{equation}
    \depth{(\im\delta_{j-2})_{\pp}}\ge\min\bigl(\depth K^{j-1}_{\pp},1\bigr)\ge1.
  \end{equation}
  \Cref{thm:seq-Lk} (with \(k=j\) and~\(L_{0}=0\)) implies
  \(\depth M_{\pp}\ge j=\min(j,d)\).

  From \Cref{thm:char-syzygy}
  we conclude that \(M\) is a \(j\)-th syzygy.

  \smallskip
  
  \(\Rightarrow\): We proceed by induction on~\(j\).
  For~\(j=1\), let \(N=\ker\iota\), so that
  \begin{equation}
    0 \to N \to M \to K^{0}
  \end{equation}
  is exact. We have to show that \(N=0\).
  Otherwise, we use \Cref{thm:local-depth-0} to obtain a prime ideal~\(\pp\lhd R\) such that \(\depth N_{\pp}=0\)
  and in particular \(N_{\pp}\ne0\). 
  This implies \(\depth M_{\pp}=0\) and therefore \(\depth R_{\pp}=0\) by \Cref{thm:char-syzygy}. 
  By~\eqref{eq:local-0}, \(K^{i}_{\pp}\) vanishes for~\(i>0\).
  Assumption~\ref{ass:exact} now implies that \(\bar K^{*}_{\pp}\) is exact,
  contrary to our assumption~\(N_{\pp}\ne0\).

  Now assume ~\(j\ge2\). By induction
  we have \(H^{i}(\bar K^{*})=0\) for~\(i\le j-3\), hence
  \begin{equation}
    \label{eq:seq-N}
    0\longrightarrow M \longrightarrow K^{0}
    \longrightarrow\cdots
    \longrightarrow K^{j-3}
    \longrightarrow \ker\delta_{j-2}
    \longrightarrow N
    \longrightarrow 0
  \end{equation}
  is exact, where \(N=\ker\delta_{j-2}/\im\delta_{j-3}\).
  As before, we assume \(N\ne0\) and
  choose a~\(\pp\lhd R\) such that \(\depth N_{\pp}=0\).

  We claim that \(\depth R_{\pp}\ge j\):
  Otherwise \(\bar K^{*}_{\pp}\) would be exact
  at the positions~\(i\ge j\) by~\eqref{eq:local-0}
  and at the positions~\(i\le j-3\) by induction,
  hence everywhere by assumption~\ref{ass:exact},
  which contradicts our hypothesis~\(N_{\pp}\ne0\).

  As a consequence, we again have the inequalities~\eqref{eq:depth-ge-ji}
  and also
  \begin{equation}
    \depth{(\ker\delta_{j-2})_{\pp}}\ge \min\bigl( \depth K^{j-2}_{\pp},1 \bigr) \ge 1.
  \end{equation}
  \Cref{thm:seq-Lk} (with~\(k=j-1\)) applied to~\eqref{eq:seq-N} yields
  \(\depth N_{\pp}\ge1\), which is again a contradiction.
  Hence \(N=0\), as was to be shown.
\end{proof}

\section{Topological preliminaries}
\label{sec:top-prelim}

\subsection{A converse to the Leray--Hirsch theorem}

In this section we consider singular cohomology with coefficients in a field~\(\kk\).

Recall that an action of a group~\(G\) on a vector space~\(M\) is called nilpotent
if there is a finite \(G\)-stable filtration
\begin{equation}
  \label{eq:def-nilpotent}
  0 = M_{-1} \subset M_{0} \subset \dots \subset M_{k} = M
\end{equation}
such that \(G\) acts trivially on all subquotients~\(Q_{s}=M_{s}/M_{s-1}\).

\begin{proposition}
  \label{thm:nhz}
  Let \(F\hookrightarrow E\to B\) be a Serre fibration. Assume that \(B\) is connected and of finite type
  and that the action of~\(G=\pi_{1}(B)\) on~\(H^{*}(F)\) is nilpotent.
  Then \(H^{*}(E)\) is free over~\(H^{*}(B)\) if and only if the restriction map~\(H^{*}(E)\to H^{*}(F)\) is surjective.
  In this case \(G\) acts trivially on~\(H^{*}(F)\), and there is an isomorphism of \(H^{*}(B)\)-modules
  \( 
    H^{*}(E) \cong H^{*}(F)\otimes H^{*}(B)
  \). 
\end{proposition}

The ``if'' part is the usual Leray--Hirsch theorem. The converse follows from the cohomological version
of Dwyer's strong convergence result for the Eilenberg--Moore spectral sequence \cite{Dwyer:1974}.
Below we give a more elementary proof that is based on the Serre spectral sequence
and free of delicate convergence issues.
It has the additional advantage that it generalizes
to all situations where one has a spectral sequence analogous to Serre's. This includes equivariant cohomology
with compact supports and/or twisted coefficients that we are going to consider later.
For bundles over the classifying space of a torus, an argument similar to ours appears in~\cite[Thm.~1.7]{Allday:1975}.
The proof given there generalizes to our setting if one assumes that the \(G\)-action on~\(H^{*}(F)\) is trivial.

\begin{proof}
  We have justified the ``if'' part above. That \(G\) acts trivially on~\(H^{*}(F)\) follows
  from the fact that the image of the restriction map is contained in the \(G\)-invariants.

  For the ``only if'' direction we consider the Serre spectral sequence of the fibration, whose second page is of the form
  \begin{equation}
    \label{eq:serre}
    E_{2}^{p,q} = H^{p}(B;\HH^{q}(F)).
  \end{equation}
  Since the edge homomorphism~\(E_{\infty}^{0,*}\to H^{*}(F)\) coincides with the restriction map,
  we have to show that \(E_{2}^{0,*}=H^{*}(F)^{G}\) is all of~\(H^{*}(F)\) and that there are no higher differentials.
  
  Otherwise, let \(q\ge0\) be minimal such that
  \begin{enumarabic}
  \item \label{case1} \(G\) does not act trivially on~\(H^{q}(F)\), or
  \item \label{case2} \(G\) acts trivially on~\(H^{q}(F)\), but there is a differential
    \begin{equation}
      d_{r}^{p,q+r-1}\colon E_{r}^{p,q+r-1} \to E_{r}^{p+r,q}
    \end{equation}
    ending on the \(q\)-th row for some~\(r\ge2\).
  \end{enumarabic}
  Then for any~\(s<q\) and any~\(p\ge0\) we have
  \begin{equation}
    \label{eq:serre-bis}
    E_{\infty}^{p,s} = E_{2}^{p,s} = H^{p}(B)\otimes H^{s}(F).
  \end{equation}
  In the second case we additionally know that \(E_{2}^{p,q}\) is still of this form;
  in the first case we have \(E_{\infty}^{0,q} = E_{2}^{0,q} = H^{q}(F)^{G}\).

  We make the following general observation for
  an \(\N\)-graded module~\(M^{*}\) over~\(R^{*}=H^{*}(B)\).
  A graded \(\kk\)-vector subspace~\(N^{*}\subset M^{*}\)
  is contained in a minimal vector subspace generating \(M^{*}\) over~\(R^{*}\) if and only if \(N^{*}\) projects injectively to~\(M^{*}/R^{>0}\).
  (Here we are using that \(B\) is connected.)
  Moreover, if this holds and if additionally \(M^{*}\) is free over~\(R^{*}\),
  then the multiplication map~\(R^{*} \otimes N^{*} \to M^{*}\) is injective.

  Now for the second case above,
  let \(N^{*}\subset H^{*}(E)\) be the subspace spanned by representatives of a basis for~\(E_{\infty}^{0,s}=H^{s}(F)\) with~\(s\le q\).
  It injects into~\(H^{*}(E)/R^{>0}\), but the non-zero differential ending on the \(q\)-th row means that
  the multiplication map~\(R^{*} \otimes E_{\infty}^{0,q} \to E_{\infty}^{*,q}\) is not injective.
  Because the map~\(R^{*} \otimes N^{\le q-1} \to H^{*}(E)\) is an isomorphism
  onto the \((q-1)\)-st piece of the filtration of~\(H^{*}(E)\) inducing \(E_{\infty}^{*,*}\),
  this implies that the map~\(R^{*} \otimes N^{*} \to H^{*}(E)\)
  cannot be injective, either. Hence \(H^{*}(E)\) is not free over~\(R^{*}\).
  
  We turn to the first case.
  Assume that there is a subspace~\(W\subset E_{\infty}^{0,q}\) such that the map~\(R^{*}\otimes W\to E_{\infty}^{*,q}\)
  is not injective. Defining \(N^{s}\) for~\(s<q\) as before and \(N^{q}\) as a lift of~\(W\) to~\(H^{*}(E)\), we can argue as in the previous case
  to conclude that there is no~\(q\) such that \ref{case1} or~\ref{case2} holds.

  \def\Etilde{F}
  To complete the proof,
  it remains to show that such a subspace~\(W\subset E_{\infty}^{0,q}=H^{*}(F)^{G}\) exists.
  We consider a finite filtration of~\(H^{q}(F)\) of the form~\eqref{eq:def-nilpotent}.
  It gives rise to a spectral sequence~\(\Etilde_{*}^{*,*}\) converging to the row~\(E_{2}^{*,q} = H^{*}(B;\HH^{q}(F))\)
  of the previous spectral sequence. Because \(G\) acts trivially on the quotients~\(Q_{s}\), we have
  \begin{equation}
    \Etilde_{2}^{p-s,s} = H^{p}(B; \QQ_{s}) = H^{p}(B)\otimes Q_{s}.
  \end{equation}
  The \(G\)-action on~\(H^{*}(F)\) is non-trivial, however, so
  there must be some non-zero higher differential.
  We take the minimal~\(s\ge0\) such that there is a differential ending in the \(s\)-th row.
  Again the differential cannot end at~\(p=0\), so that \(H^{0}(B;\QQ_{t})\) for~\(t\le s\)
  survives to~\(\Etilde_{\infty}\) and therefore gives rise to a subspace~\(W\subset E_{2}^{0,q}=E_{\infty}^{0,q}\).
  As before, the existence of a higher differential implies that the map~\(R^{*}\otimes W\to E_{2}^{*,q}\)
  cannot be injective, and the same applies to the map~\(R^{*}\otimes W\to E_{\infty}^{*,q}\).
\end{proof}

\subsection{Torus-equivariant homology and cohomology}
\label{sec:torus-equiv}

We consider the same class of spaces as in~\citeorbitsone{Sec.~3.1} and~\citeorbitsfour{Sec.~2.1}.
This means that all spaces are assumed to be Hausdorff, second-countable, locally compact, locally contractible
and of finite covering dimension,
hence also separable and metrizable.
This includes topological (in particular, smooth) manifolds, orbifolds, complex algebraic varieties
as well as countable, finite-dimensional, locally finite CW~complexes. In fact, one can deduce
from~\cite[Prop.~3.4, Thms.~9.5~\&~12.5\,(2)]{Murayama:1983} that
the above assumptions hold for countable, finite-dimensional, locally finite \(G\)-CW complexes for any compact Lie group~\(G\).
Whenever we consider a subspace of some topological space, we assume it to be locally contractible, too.
See~\citeorbitsone{Rem.~4.7} for a way to avoid this condition.

Unless indicated otherwise, all (co)homology and all tensor products are taken over a field~\(\kk\) of characteristic~\(p>0\).
We grade all complexes cohomologically. In case of a space~\(X\) for instance, an element in~\(H_{n}(X)\) has degree~\(-n\).
For any graded module~\(M^{*}\), we write \(M^{*}[m]\) for the same graded module
with degrees shifted upwards by~\(m\in\Z\). Since we grade homology negatively,
an element in~\(H_{n}(X)\) has degree~\(m-n\) in~\(H_{n}(X)[m]\), for example.

Let \(T\cong(S^{1})^{r}\) be a torus of rank~\(r\ge0\). We write \(\RT=H^{*}(BT)\)
and \(\HT^{*}(A,B)\) for the Borel-equivariant singular cohomology of a \(T\)-pair~\((A,B)\).
By our assumptions on spaces, the latter is isomorphic to equivariant Alexander--Spanier cohomology,
and the same holds non-equivariantly.

We use the definition of \(T\)-equivariant homology~\(\hHT_{*}(A,B)\)
given in~\cite[Sec.~3.3]{AlldayFranzPuppe:orbits1}. 
We recall that \(\hHT_{*}(X)\) is not the homology of the Borel construction~\(X_{T}=ET\times_{T}X\).
On the other hand, it satisfies a localization theorem as well as universal coefficient theorems,
see~\cite[Prop.~3.5]{AlldayFranzPuppe:orbits1} and~\cite[Prop.~2.5]{AlldayFranzPuppe:orbits4}.
There is a pairing
\begin{equation}
  \label{eq:pairing-hHT-HT}
  \HT^{*}(X) \otimes \hHT_{*}(X) \to \RT
\end{equation}
coming from the algebraic duality between equivariant chains and cochains.

As in~\cite{AlldayFranzPuppe:orbits4}, we consider several other variants of torus-equivariant (co)ho\-mol\-ogy.
We write \(\HTc^{*}(A,B)\) for the equivariant cohomology with compact supports of the closed \(T\)-pair~\((A,B)\)
and \(\smash{\hHTc_{*}(A,B)}\) for its equivariant homology with closed supports, see~\cite[Sec.~2.3]{AlldayFranzPuppe:orbits4}.

Moreover, in case \((A,B)\) is contained in a \(\kk\)-homology manifold~\(X\),
we also consider equivariant (co)ho\-mol\-ogy with twisted coefficients, denoted by~\(\HT(A,B;\kktilde)\) and~\(\hHT_{*}(A,B;\kktilde)\).
The pairing~\eqref{eq:pairing-hHT-HT} extends to all these variants.
We refer to~\cite[Secs.~2.5~\&~2.6]{AlldayFranzPuppe:orbits4} for the definitions.
Note that in this case we assume in particular that \(X\) has an orientation cover
unless \(p=2\), in which case there is no difference between twisted and constant coefficients.

We write for example~\(\HT^{*}(X,A\AA)\) in situations where we consider both
constant (\(\ell=\kk\)) and twisted (\(\ell=\kktilde\)) coefficients.
In the latter case we always assume \((A,B)\) to be contained in a homology manifold~\(X\).
Twisted coefficients may also be combined with other supports.

We assume throughout that the \(\RT\)-modules~\(\HT^{*}(A,B)\) and \(\hHT_{*}(A,B)\) are finitely generated.
This is equivalent to \(H^{*}(A,B)\) being finite-dimensional over~\(\kk\). In the case of twisted coefficients,
we assume that both \(H^{*}(A,B)\) and~\(H^{*}(A,B;\kktilde)\) are finite-dimensional.
The same applies to the other pair of supports.
Unlike in the case of rational coefficients treated in~\cite{AlldayFranzPuppe:orbits1},~\cite{AlldayFranzPuppe:orbits4},
it is not necessary in characteristic~\(p\) to require \(T\)-spaces to have only finitely many connected orbit types.

We finally record the following application of \Cref{thm:syz-prod-alg}.

\begin{proposition}
  \label{thm:syz-prod}
  Let \(X'\) and~\(X''\) be spaces with actions of the tori~\(T'\) and~\(T''\), respectively,
  so that \(T=T'\times T''\) acts on~\(X'\times X''\). Then
  \begin{equation*}
    \syzord_{R}\HT^{*}(X'\times X'') = \min\bigl(\syzord_{R'} H_{T'}^{*}(X'),\syzord_{R''} H_{T''}^{*}(X'')\bigr),
  \end{equation*}
  and analogously for cohomology with compact supports.
\end{proposition}

\section{Induction and restriction}
\label{sec:ind-red}

For the rest of this paper, \(G\cong(\Z_{p})^{r}\) denotes a \(p\)-torus of rank~\(r\).
We think of~\(G\) as the maximal \(p\)-torus contained in the torus~\(T\cong(S^{1})^{r}\).
We take the total space~\(ET\) of the universal \(T\)-bundle as~\(EG\), so that we have a fibre bundle
\begin{equation}
  \label{eq:bundle-BG-BT}
  \begin{tikzcd}[column sep=scriptsize]
    T/G \arrow[hook]{r}{\iota} & ET/G=BG \arrow{r}{\pi} & BT
  \end{tikzcd}
\end{equation}
inducing the maps
\begin{equation}
  \label{eq:bundle-BG-BT-cohom}
  \begin{tikzcd}[column sep=scriptsize]
    \Lambda=H^{*}(T/G) & \arrow{l}[above]{\iota^{*}} \RG = H^{*}(BG) & \arrow{l}[above]{\pi^{*}} \RT = H^{*}(BT)
  \end{tikzcd}
\end{equation}
in cohomology.
Since \(T/G\) is again a torus, its cohomology
is the (strictly) exterior algebra~\(\Lambda=\bigwedge(x_{1},\dots,x_{r})\) on generators of degree~\(1\).
Moreover, \(\RT=\kk[t_{1},\dots,t_{r}]\) is a polynomial algebra on generators of degree~\(2\).
Given that we work over a field of characteristic~\(p\), the map~\(\iota^{*}\) is surjective.
The Leray--Hirsch theorem implies that \(\pi^{*}\) is injective, that the kernel of~\(\iota^{*}\)
is the ideal generated by the image of the maximal homogeneous ideal~\((t_{1},\dots,t_{r})\lhd\RT\) under~\(\pi^{*}\)
and that there is an isomorphism of \(\RT\)-modules
\begin{equation}
  \label{eq:iso-RG}
  \RG = \Lambda\otimes\RT.
\end{equation}
Since \(\Lambda\) is free as a graded commutative algebra for~\(p\ne2\), the section~\(\Lambda\to\RG\)
can be chosen multiplicative, turning \eqref{eq:iso-RG} into an isomorphism of \(\RT\)-algebras in this case.
In contrast, for~\(p=2\) we have \(x_{i}^{2}=t_{i}\) in the polynomial ring~\(\RG=\kk[x_{1},\dots,x_{r}]\).
We remark in passing that the isomorphism~\eqref{eq:iso-RG} is actually natural with respect to morphisms of \(p\)-tori.

Our general strategy is to deduce results about \(G\)-equivariant cohomology from those about \(T\)-equivariant cohomology.
For any \(G\)-space~\(X\) we therefore consider the induced \(T\)-space
\begin{equation}
  \label{eq:def-tildeX}
  \indT{X} = T\times_{G} X.
\end{equation}
It is the total space of a bundle with fibre~\(X\) associated to the principal \(G\)-bundle \(G\hookrightarrow T\to T/G\).
Note that the assignment~\(A\mapsto\indT{A}\) sets up a bijection between the \(G\)-stable (open, closed) subsets~\(A\subset X\)
and the \(T\)-stable (open, closed) subsets of~\(\indT{X}\).
We pick once and for all an orientation~\(\bar o\in H_{r}(T/G)\).

\begin{lemma}
  Assume that \(X\) is an \(n\)-dimensional homology manifold.
  Then \(\indT{X}\) is a homology manifold of dimension~\(n+r\).
\end{lemma}

\begin{proof}
  The Künneth theorem for local cohomology implies that the product~\(T\times X\)
  is a homology manifold of dimension~\(n+r\), hence also the quotient~\(\indT{X}\) by the freely acting group~\(G\).
\end{proof}

\begin{lemma}
  \label{thm:cover}
  Let \(X\) be a non-orientable \(\kk\)-homology manifold with a \(G\)-action. Then the action
  lifts to the orientation cover~\(X'\to X\), where it commutes with the deck transformation.
\end{lemma}

\begin{proof}
  For connected~\(X\) this is analogous to the manifold case proven in~\cite[Cor.~I.9.4]{Bredon:1972}.
  Otherwise we may assume that \(G\) acts transitively on the connected components of~\(X\).
  Let \(K\subset G\) be the subgroup stabilizing a given component~\(X_{0}\). Then
  \begin{equation}
    X' = G\times_{K} X_{0}' \to G\times_{K} X_{0} = X
  \end{equation}
  is the orientation cover, and the claim follows from the connected case, applied to the \(K\)-action on~\(X_{0}\).
\end{proof}

We get as a consequence that if \(X'\to X\) is an orientation cover of the homology manifold~\(X\),
then \(\indT{X'}=T\times_{G} X'\to\indT{X}=T\times_{G} X\) is again an orientation cover.

\begin{lemma}
  \label{thm:ss-tilde-X}
  Let \((A,B)\) be a \(G\)-pair.
  There is a canonical map
  \begin{equation*}
    H_{*}(\indT{A},\indT{B}\AA)[-r] \to H_{*}(A,B\AA)^{G},
  \end{equation*}
  natural in~\((A,B)\), and analogously for homology with closed supports.
  The map is an isomorphism between the highest non-vanishing degrees.

  In particular, any \(G\)-invariant orientation~\(o\in\hHc_{n}(X\AA)\) of an \(n\)-dimensional homology manifold~\(X\)
  lifts uniquely to an orientation~\(\indT{o}\in\hHc_{n+r}(\indT{X}\AA)\).%
\end{lemma}

Recall that we write degree shifts cohomologically while grading homology negatively,
so that an element in~\(H_{n}(\indT{A},\indT{B}\AA)[-r]\) has homological degree~\(n+r\).

\begin{proof}
  Let us start with the first claim in the version for homology with compact supports and constant coefficients.
  We consider the Serre spectral sequence for the bundle~\eqref{eq:bundle-X-tilde} with second page
  \begin{equation}
    E^{2}_{p,q} = H_{p}(T/G; \HH_{q}(A,B)),
  \end{equation}
  where \(\HH_{*}(-)\) denotes local coefficients.
  Using the canonical decomposition of the torus~\(T/G\) into~\(2^{r}\) cells,
  we can see that the rightmost column is of the form
  \begin{equation}
    E^{2}_{r} = \spanofel{\bar o} \otimes H_{*}(A,B)^{G},
  \end{equation}
  and \(E^{\infty}_{r}\) is a subspace of it. Hence we get a map
  \begin{equation}
    H_{*}(\indT{A},\indT{B})[-r] \to E^{\infty}_{r} \hookrightarrow E^{2}_{r}
    = \spanofel{\bar o} \otimes H_{*}(A,B)^{G} \cong H_{*}(A,B)^{G},
  \end{equation}
  where the last isomorphism depends on the orientation~\(\bar o\) (of cohomological degree~\(-r\)).
  Because \(E^{2}_{r,n}=E^{\infty}_{r,n}\) in the highest non-vanishing degree~\(n\),
  the composition is an isomorphism in this case.

  For homology with closed supports
  we either use an appropriate version of the Serre spectral sequence,
  or we dualize the result for compact supports to get a map
  \begin{equation}
    H^{*}(\indT{U},\indT{V})[-r] \leftarrow H^{*}(U,V)_{G}
  \end{equation}
  for any \(G\)-stable open neighbourhood pairs~\((U,V)\) of~\((A,B)\) in the \(G\)-space~\(X\) such that \(X\setminus V\) is compact;
  the subscript~``\(G\)'' on the right-hand side above denotes the covariants.
  Note that neighbourhood pairs of the form~\((\indT{U},\indT{V})\) of~\((\indT{A},\indT{B})\)
  are cofinal among all neighbourhood pairs~\((\hat U,\hat V)\) such that \(\indT{X}\setminus\hat V\) is compact.
  Taking the limit over all such pairs as in~\citeorbitsfour{eq.~(2-7)} therefore gives the map
  \begin{equation}
    \Hc^{*}(\indT{A},\indT{B})[-r] \leftarrow \Hc^{*}(A,B)_{G},
  \end{equation}
  which is again an isomorphism in top degree.
  Dualizing back to homology gives the desired result.

  For twisted coefficients we pass to the orientation cover~\(X'\to X\) (\Cref{thm:cover})
  and use the decomposition
  \begin{equation}
    H_{*}(A',B') = H_{*}(A,B) \oplus H_{*}(A,B;\kktilde)
  \end{equation}
  into eigenspaces of the \(G\)-equivariant deck transformation,
  compare~\citeorbitsfour{Sec.~2.6}, and the same for~\((\indT{A},\indT{B})\).
  For closed supports we proceed analogously.

  Because the orientation~\(o\in\hHc_{n}(X\AA)\) of~\(X\)
  lives in top degree and is \(G\)-invariant, it lifts uniquely to an orientation of~\(\indT{X}\),
  proving the final claim.
\end{proof}

We state the following results for the usual pair of supports. They are equally valid for the other pair.

\begin{proposition}
  \label{thm:ind-HG-HT}
  Let \((A,B)\) be a \(G\)-pair in~\(X\).
  There is a natural isomorphism of \(\RT\)-algebras
  \begin{equation*}
    \HG^{*}(A,B\AA) = \HT^{*}(\indT{A},\indT{B}\AA).
  \end{equation*}
\end{proposition}

\begin{proof}
  Since~\(ET\times_{T}\indT{X}=ET\times_{G}X\),
  this is immediate for closed supports and constant coefficients.
  One extends it to compact supports by taking limits
  and finally to twisted coefficients by passing to the orientation cover as before.
\end{proof}

Conversely, we can consider any \(T\)-space as a \(G\)-space by restricting the action.

\begin{proposition}
  \label{thm:red-HG-HT}
  Let \((A,B)\) be a \(T\)-pair.
  There are natural isomorphism
  \begin{equation*}
    \HG^{*}(A,B\AA) = \RG \otimes_{\RT} \HT^{*}(A,B\AA) = \Lambda \otimes \HT^{*}(A,B\AA).
  \end{equation*}
  The first isomorphism is one of \(\RT\)-algebras, as is the second for~\(p\ne2\).
  In general, the second one is only one of \(\RT\)-modules for~\(p=2\).
\end{proposition}

\begin{proof}
  Assume that \((A,B)\) is a pair in the \(T\)-space~\(X\), and consider again closed supports and constant coefficients first.
  Given that \(\RG\) is free over~\(\RT\),
  the claim then is a consequence of the Eilenberg--Moore spectral sequence applied to the pull-back
  \begin{equation}
    \begin{tikzcd}
      ET\times_{G}X \arrow{d} \arrow{r} & ET/G \mathrlap{{}=BG} \arrow{d} \\
      ET\times_{T}X \arrow{r} & BT \mathrlap{,}
    \end{tikzcd}
  \end{equation}
  see~\cite[Prop.~II.4.3, Cor.~II.4.4]{Smith:1967} for~\(B=\emptyset\).
  The extension to the other cases is done as in the preceding proofs.
\end{proof}

\begin{remark}
  It follows from \Cref{thm:red-HG-HT} that for~\(p\ne2\) we have an isomorphism of \(\RT\)-algebras
  \begin{equation}
    \HT^{*}(A,B\AA) = \kk \otimes_{\Lambda} \HG^{*}(A,B\AA).
  \end{equation}
  We do not know of a way to recover \(T\)-equivariant from \(G\)-equivariant cohomology in the case~\(p=2\).
\end{remark}

\begin{corollary}
  \label{thm:red-HG-HT-alg}
  Let \((X,Y)\) be a \(T\)-pair. We have the following equivalences
  regarding the \(\RT\)-modules~\(\HT^{*}(A,B\AA)\) and~\(\HG^{*}(A,B\AA)\).
  \begin{enumroman}
  \item
    \label{thm:red-HG-HT-cm}
    \(\HG^{*}(A,B\AA)\) is Cohen--Macaulay of projective dimension~\(i\) if and only if so is \(\HT^{*}(A,B\AA)\).
  \item
    \label{thm:red-HG-HT-syz}
    \(\HG^{*}(A,B\AA)\) is a \(j\)-th syzygy if and only if so is \(\HT^{*}(A,B\AA)\).
  \end{enumroman}
\end{corollary}

\begin{proof}
  As an \(\RT\)-module, \(\HG^{*}(A,B\AA)\) is the direct sum of \(2^{r}\)~copies of~\(\HT^{*}(A,B\AA)\)
  by \Cref{thm:red-HG-HT}. The claims follow because the direct sum of finitely many non-zero \(\RT\)-modules is Cohen--Macaulay
  of a certain dimension (or a certain syzygy) if and only if this holds for each summand.
\end{proof}

\section{The Cohen--Macaulay property of the orbit filtration}
\label{sec:orbit-cm}

Given a space~\(X\) with an action of the torus~\(T\) or just of the \(p\)-torus~\(G\subset T\),
we consider the \newterm{orbit filtration}
\begin{equation}
  \label{eq:def-orbit-filt}
  \emptyset = \Xp{-1} \subset \Xp{0} \subset \dots \subset \Xp{r} = X
\end{equation}
where \(\Xp{i}\) is the set of all points~\(x\in X\) whose \(G\)-orbit has at most~\(p^{i}\) elements.
By our assumption on spaces stated in \Cref{sec:torus-equiv}, all \(X_{i}\) are closed in~\(X\)
and required to be locally contractible. The latter condition holds, for instance,
for smooth and algebraic actions as well as for \(G\)-CW~complexes and \(T\)-CW~complexes in general.

It moreover follows from the localization theorem for \(G\)-equivariant cohomology
that all~\(\HG^{*}(X_{i},X_{j})\) with~\(j<i\) are finitely generated over~\(\RT\) because so is \(\HG^{*}(X)\).
(The absolute case is a consequence of~\cite[Prop.~4.1.14]{AlldayPuppe:1993}. Together with the long exact sequence
of the pair~\((X_{i},X_{j})\) it implies the relative statement.) The same holds for twisted coefficients and/or other supports.

Note that in case of a torus action the filtration~\eqref{eq:def-orbit-filt} differs from the orbit filtration studied in~\cite{AlldayFranzPuppe:orbits1}
and~\cite{AlldayFranzPuppe:orbits4}. This is imposed by the different choice of coefficients~\(\kk\).

The starting point of our theory is the following result.
In terms of the terminology introduced in~\cite[Sec.~6.1]{Franz:nonab}
it means that the filtration~\((\Xp{i})\) is Cohen--Macaulay (for coefficients in~\(\kk\)).

\begin{proposition}
  \label{thm:Xi-CM-G}
  Let \(X\) be a \(G\)-space. The \(\RT\)-module~\(\HG^{*}(\Xp{i},\Xp{i-1}\AA)\)
  is zero or Cohen--Macaulay of projective dimension~\(i\) for any~\(0\le i\le r\).
\end{proposition}

\begin{proof}
  \def\KK{K} 
  \def\LL{L} 
  We write the proof down for cohomology with closed supports and indicate additional steps for compact supports.
  
  It is enough to prove the claim for constant coefficients because
  for the orientation cover~\(X'\) we have the splitting
  \begin{equation}
    \label{eq:HG-splitting}
    \HG^{*}(\Xpp{i},\Xpp{i-1}) = \HG^{*}(\Xp{i},\Xp{i-1}) \oplus \HG^{*}(\Xp{i},\Xp{i-1};\kktilde),
  \end{equation}
  compare~\citeorbitsfour{eq.~(2-44)} for the torus case.
  If the direct sum~\eqref{eq:HG-splitting} is zero or Cohen--Macaulay of projective dimension~\(i\),
  then so is each summand.
  
  First we consider the case where all isotropy groups occurring in~\(X\)
  contain some fixed \(p\)-torus~\(\KK\subset G\) of rank~\(r-i\).
  Let \(\LL\subset G\) be a complement to~\(\KK\) inside~\(G\) and write \(Z=\Xp{i-1}\).
  By assumption, \(\LL\) operates freely on~\(X\setminus Z\) while \(\KK\) acts trivially.
  For cohomology with closed supports we therefore get
  \begin{equation}
    \label{eq:HGXZ}
    \HG^{*}(X,Z) = H^{*}(B\KK)\otimes H_{\LL}^{*}(X,Z) = H^{*}(B\KK)\otimes H^{*}(X/L,Z/L),
  \end{equation}
  compare~\cite[Prop.~3.10.9]{AlldayPuppe:1993}.
  We extend this to cohomology with compact supports by taking the limit over all \(G\)-stable open subsets~\(U\subset X\) containing~\(Z\).
  
  Consider the increasing filtration of~\(M^{*}=H^{*}(X/L,Z/L)\) by degree as well as the filtration of~\(\HG^{*}(X,Z)\) induced via~\eqref{eq:HGXZ}.
  Both are finite since \(X/\LL\) is again of finite covering dimension.
  While \(H^{*}(B\LL)\) may not act trivially on~\(M^{*}\),
  it necessarily does so on each piece~\(Q_{s}^{*}\) of the associated graded module.
  Each piece~\(H^{*}(B\KK)\otimes Q_{s}^{*}\) of the associated graded module for~\(\HG^{*}(X,Z)\)
  therefore is zero or Cohen--Macaulay of projective dimension~\(i\).
  It follows by induction from \Cref{thm:CM-Ext} and the long exact sequence for~\(\Ext\) that so is \(\HG^{*}(X,Z)\).
  
  We now turn to the general case.
  Let \((G_{\alpha})\) be the finitely many subgroups of~\(G\) of rank~\(r-i\).
  For each of them define \(X_{\alpha}= \{x\in X\mid G_{x}\supset G_{\alpha}\}\) and \(Z_{\alpha}=(X_{\alpha})_{i-1}\).
  Writing \(Y_{\alpha}=X_{\alpha}/Z_{\alpha}\), we have isomorphisms
  \begin{equation}
    \HG^{*}(\Xp{i},\Xp{i-1}) = \HG^{*}(\Xp{i}/\Xp{i-1},*) = \HG^{*}\Bigl(\,\bigvee_{\alpha} Y_{\alpha},*\Bigr)
    = \bigoplus_{\alpha} \HG^{*}(X_{\alpha}, Z_{\alpha}).
  \end{equation}
  Each summand is zero or Cohen--Macaulay of projective dimension~\(i\) by the previous discussion, whence the claim.
\end{proof}

\begin{corollary}
  \label{thm:Xi-CM-T}
  Let \(X\) be a \(T\)-space. The \(\RT\)-module~\(\HT^{*}(\Xp{i},\Xp{i-1}\AA)\)
  is zero or Cohen--Macaulay of projective dimension~\(i\) for any \(0\le i\le r\).
\end{corollary}

\begin{proof}
  Combine \Cref{thm:Xi-CM-G} with \Cref{thm:red-HG-HT-alg}\,\ref{thm:red-HG-HT-cm}.
\end{proof}

\section{Equivariant homology for \texorpdfstring{\(p\)}{p}-tori}
\label{sec:equiv-hom}

\subsection{Definition and first properties}
\label{sec:def-equiv-hom}

Let \((A,B)\) be a \(G\)-pair in a \(G\)-space~\(X\).
Recall from~\eqref{eq:def-tildeX} that the induced \(T\)-space~\(\indT{X}\) is the total space of the bundle
\begin{equation}
  \label{eq:bundle-X-tilde}
  X \hookrightarrow \indT{X} \stackrel\pi\longrightarrow T/G
\end{equation}
associated to the principal \(G\)-bundle~\(T\to T/G\).

We define the \newterm{\(G\)-equivariant homology} of~\((A,B)\) as the \(R\)-module
\begin{equation}
  \label{eq:def-hHG}
  \hHG_{*}(A,B\AA) = \hHT_{*}(\indT{A},\indT{B}\AA)[r].
\end{equation}
The degree shift takes the difference between the dimensions of~\(X\) and~\(\indT{X}\) into account,
see for example the computation of~\(\hHG_{*}(pt)\) in~\eqref{eq:hHG-pt} below.
The pairing~\eqref{eq:pairing-hHT-HT} translates to a pairing
\begin{equation}
  \label{eq:pairing-hHG-HG}
  \HG^{*}(X\AA) \otimes \hHG_{*}(X\AA) \to \RT[r].
\end{equation}
The \(G\)-equivariant homology~\(\hHGc_{*}(A,B\AA)\) with compact supports is defined analogously.
Note that we do not define an \(\RG\)-module structure on equivariant homology
(but see \Cref{sec:remarks-2} for an alternative approach in the case of \(2\)-tori).

There are spectral sequences of \(\RT\)-modules
\begin{align}
  \label{eq:serre-ss-cohom}
  E_{2} = H^{*}(\indT{A},\indT{B}\AA)\otimes\RT &\Rightarrow \HG^{*}(A,B\AA), \\
  \label{eq:serre-ss-hom}
  E_{2} = H_{*}(\indT{A},\indT{B}\AA)[r]\otimes\RT &\Rightarrow \hHG_{*}(A,B\AA)
\end{align}
and similarly for the other pair of supports,
see~\citeorbitsone{eqs.~(3.5)~\&~(3.7)} and~\citeorbitsfour{Prop.~2.3}.
The edge homomorphism for the first spectral sequence
together with the restriction to the fibre of the bundle~\eqref{eq:bundle-X-tilde}
gives the map
\begin{equation}
  \HG^{*}(A,B\AA) = \HT^{*}(\indT{A},\indT{B}\AA)
  \to H^{*}(\indT{A},\indT{B}\AA) \to H^{*}(A,B\AA)^{G},
\end{equation}
which agrees with the canonical restriction map~\(\HG^{*}(A,B\AA)\to H^{*}(A,B\AA)^{G}\).
Similarly, combining the edge homomorphism for the second spectral sequence
with \Cref{thm:ss-tilde-X}, we get a restriction map in \(G\)-equivariant homology,
\begin{equation}
  \label{eq:restr-equiv-hom}
  \hHG_{*}(A,B\AA) = \hHT_{*}(\indT{A},\indT{B}\AA)[-r]
  \to H_{*}(\indT{A},\indT{B}\AA)[-r] \to H_{*}(A,B\AA)^{G},
\end{equation}
natural in the pair~\((A,B)\).

In~\citeorbitsone{Prop.~3.5},~\citeorbitsfour{Rem.~2.15} we related torus-equivariant homology and cohomology
(for arbitrary field coefficients) via universal coefficient spectral sequences.
The same result holds in our present setting.

\begin{proposition}
  \label{thm:uct-G}
  Let \((A,B)\) be a \(G\)-pair. Then there are spectral sequences
  \begin{align*}
    E^{2}_{i} = \Ext_{\RT}^{i}(\HG^{*}(A,B\AA),\RT[r]) &\Rightarrow \hHG_{*}(A,B\AA), \\
    E_{2}^{i} = \Ext_{\RT}^{i}(\hHG_{*}(A,B\AA),\RT[r]) &\Rightarrow \HG^{*}(A,B\AA)
  \end{align*}
  natural in~\((A,B)\).
\end{proposition}

\begin{proof}
  This is a reformulation of the result for tori, based on \Cref{thm:ind-HG-HT}
  and the definition of \(G\)-equivariant homology. The degree shifts are
  dictated by the one in~\eqref{eq:def-hHG}.
\end{proof}

In particular, for~\(X=pt\) a point we have an isomorphism of \(\RT\)-modules
\begin{equation}
  \label{eq:hHG-pt}
  \hHG_{*}(pt) = \Hom_{\RT}(\RG,\RT[r]) = \Hom_{\RT}(\Lambda\otimes\RT,\RT)[r] = \Lambda\otimes\RT = \RG
\end{equation}
because the \(\kk\)-dual of~\(\Lambda=H^{*}(T/G)\) from~\eqref{eq:bundle-BG-BT-cohom} is \(H_{*}(T/G)\cong\Lambda[-r]\).

\smallskip

We now turn to the orbit filtration defined in \Cref{sec:orbit-cm}.

\begin{proposition}
  \label{thm:hHT-hHG-CM}
  Let \(0\le i\le r\).
  \begin{enumroman}
  \item For any \(G\)-space~\(X\), the \(\RT\)-module~\(\hHG_{*}(\Xp{i},\Xp{i-1}\AA)\)
    is zero or Cohen--Macaulay of projective dimension~\(i\).
  \item For any \(T\)-space~\(X\), the \(\RT\)-module~\(\hHT_{*}(\Xp{i},\Xp{i-1}\AA)\)
    is zero or Cohen--Macaulay of projective dimension~\(i\).
  \end{enumroman}
\end{proposition}

\begin{proof}
  In light of \Cref{thm:cm-ext} and the universal coefficient theorem, 
  the claims follow from \Cref{thm:Xi-CM-G} and \Cref{thm:Xi-CM-T}.
\end{proof}

\begin{remark}
  An isomorphism of \(\RT\)-modules analogous to \Cref{thm:red-HG-HT}
  holds in equivariant homology, that is,
  \begin{equation}
    \hHG_{*}(A,B\AA) = \RG \otimes_{\RT} \hHT_{*}(A,B\AA)
  \end{equation}
  for any \(G\)-pair~\((A,B)\).
  However, a proof of this identity requires a fairly detailed analysis of the bundle~\(\indT{X}\to G/T\).
  For our purposes \Cref{thm:hHT-hHG-CM} will be sufficient.
\end{remark}

\subsection{Poincaré duality}
\label{sec:pd}

A Poincaré duality (PD) space of formal dimension~\(n\ge0\) is a space~\(Y\) with a distinguished class~\(o\in H_{n}(Y)\)
such that the pairing
\begin{equation}
  H^{*}(Y)\otimes H^{*}(Y) \to \kk,
  \qquad
  \alpha\otimes\beta \mapsto \pair{\alpha\cup\beta,o}
\end{equation}
is perfect.

If the \(G\)-space~\(X\) we consider is a PD~space or a homology manifold,
then we assume from now on that \(G\) acts transitively on the connected components of~\(X\). In other words,
the induced \(T\)-space~\(\indT{X}\) from~\eqref{eq:def-tildeX} is assumed to be connected, as were the \(T\)-spaces
in~\cite{AlldayFranzPuppe:orbits1} and~\cite{AlldayFranzPuppe:orbits4} in the same situation.

\begin{lemma}
  \label{thm:G:comparison-PD-space}
  Assume that \(X\) is a PD~space of formal dimension~\(n\) with a \(G\)-invariant orientation~\(o\in H_{n}(X)\).
  Then \(\indT{X}\) is a canonically oriented PD~space of formal dimension~\(n+r\).
\end{lemma}

\begin{proof}
  As in the proof of \Cref{thm:ss-tilde-X}, the canonical orientation~\(\indT{o}\) of~\(\indT{X}\)
  is given in terms of the Serre spectral sequence by the tensor product of~\(\bar o\) and~\(o\) since
  \begin{equation}
    H_{n+r}(\indT{X})\cong E_{\infty}^{n,r}=E_{2}^{n,r}\cong \spanofel{\bar o}\otimes H_{n}(X)^{G}. 
  \end{equation}

  Assume for the moment that \(X\) is compact. The standard proof
  of Poincaré duality for manifolds (\cf~\cite[Sec.~3.3]{Hatcher:2002})
  shows the following:
  Suppose a space~\(Y\) has a class~\(o\in\hHc_{m}(Y)\) and
  a finite open cover~\(\UU\) which is closed under intersections.
  If for any~\(U\in\UU\) capping with the image of~\(o\) in~\(\hHc_{m}(U)\)
  gives an isomorphism
  \begin{equation}
    \label{eq:PD-U}
    \Hc^{*}(U) \to H_{m-*}(U),
  \end{equation}
  then it follows from a Mayer--Vietoris argument
  (\cf~\cite[Lemma~3.36]{Hatcher:2002})
  that capping with~\(o\) is an isomorphism for~\(X\) as well.

  In our case, consider the inverse image~\(\pi^{-1}(U)\)
  of an open set~\(U\subset T/G\)
  which trivializes the bundle~\eqref{eq:bundle-X-tilde}.
  The map~\eqref{eq:PD-U} becomes
  \begin{multline}
    \qquad\quad
    \Hc^{*}(\pi^{-1}(U)) = \Hc^{*}(U)\otimes H^{*}(X) \\
    \to H_{n+r-*}(\pi^{-1}(U)) = H_{n-*}(U)\otimes H_{r-*}(X),
    \quad\qquad
  \end{multline}
  which is an isomorphism as the image of~\(\indT{o}\) in~\(H_{n+r}(\pi^{-1}(U))\)
  is the tensor product of the restriction of~\(\bar o\) to~\(U\) and~\(o\).
  Since we can cover \(\indT{X}\) by finitely many such sets~\(\pi^{-1}(U)\),
  a reasoning analogous to the one recalled above shows that \(\indT{X}\) is a PD~space.

  If \(X\) is non-compact, then we can imitate this proof by
  considering the cohomology
  \begin{equation}
    \Hhc^{*}(\pi^{-1}(U)) = \dirlim H^{*}(\pi^{-1}(U),\pi^{-1}(V))
  \end{equation}
  where \(U\subset T/G\) is open and the direct limit is taken over all open subsets~\(V\subset U\) 
  such that \(U\setminus V\) is compact, as well as the corresponding
  homology~\(\hHhc_{*}(\pi^{-1}(U))\). There is a well-defined cup product
  \begin{equation}
    \Hhc^{*}(\pi^{-1}(U)) \otimes H^{*}(\pi^{-1}(U))
    \to \Hhc^{*}(\pi^{-1}(U)), 
  \end{equation}
  hence also a cap product
  \begin{equation}
    \Hhc^{*}(\pi^{-1}(U)) \otimes \hHc_{*}(\pi^{-1}(U)) \to H_{*}(\pi^{-1}(U)),
  \end{equation}
  \cf~\citeorbitsfour{eq.~(3-2)}.
  Moreover, if the bundle is trivial over~\(U\), then
  the image of~\(\indT{o}\)
  in~\(\hHhc_{n+r}(\pi^{-1}(U)) = \hHc_{n}(U)\otimes H_{r}(X)\)
  is the tensor product of two orientations, so that
  \begin{multline}
    \qquad\quad
    \Hhc^{*}(\pi^{-1}(U)) = \Hc^{*}(U)\otimes H^{*}(X) \\
    \to H_{n+r-*}(\pi^{-1}(U)) = H_{n-*}(U)\otimes H_{r-*}(X)
    \quad\qquad
  \end{multline}
  is again an isomorphism. Since \(\Hhc^{*}(\indT{X})=H^{*}(\indT{X})\),
  this implies as before that \(\indT{X}\) is a PD~space.
\end{proof}

\begin{proposition}
  \label{thm:PD-noncompact}
  Assume that \(X\) is an \(n\)-dimensional \(\kk \)-homology manifold.
  Then any \(G\)-invariant orientation~\(o\in\hHc_{n}(X;\kktilde)\) lifts uniquely to
  an equivariant orientation~\(o_{G}\in\hHGc_{n}(X;\kktilde)\).
  Moreover, taking the cap product with~\(o_{G}\) gives an isomorphism of \(\RT \)-modules
  (of degree~\(-n\))
  \begin{equation*}
    \HGc^{*}(X;\kktilde) \stackrel{\cap o_{G}}\longrightarrow \hHG_{*}(X).
  \end{equation*}
\end{proposition}

The equivariant cap product is defined in~\citeorbitsfour{eq.~(3-2)}.

\begin{proof}
  By \Cref{thm:ss-tilde-X}
  there is a unique lift~\(\indT{o}\in\hHc_{n+r}(\indT{X};\kktilde)\) of the orientation~\(o\).
  Translating~\citeorbitsfour{Prop.~3.2} to the present context,
  we see that \(\indT{o}\) lifts uniquely to an equivariant orientation
  \begin{equation}
    o_{G}\coloneqq\indT{o}_{T}\in\hHTc_{n+r}(X;\kktilde)=\hHGc_{n}(X;\kktilde)
  \end{equation}
  inducing the isomorphism stated above.
  From the definition~\eqref{eq:restr-equiv-hom} of the restriction map in equivariant homology
  we conclude that \(o_{G}\) is the only possible lift.
\end{proof}

The following two results concerning Poincaré--Alexander--Lefschetz duality
are translations of the corresponding result from~\citeorbitsfour{Sec.~3.2} to the case of \(G\)-actions.\footnote{%
  The rightmost column in the commutative diagram given in~\citeorbitsfour{Thm.~3.4} contains misprints
  that are corrected in \Cref{thm:PAL-A-B} above.}

\begin{theorem}
  \label{thm:PAL-A-B}
  Assume that \(X\) is an \(n\)-dimensional \(\kk\)-homology manifold, and
  let \((A,B)\) be a closed \(G\)-pair in~\(X\).
  Then there is a
  commutative diagram
  \begin{equation*}
    \let\OLDsetminus\setminus
    \def\setminus{\!\OLDsetminus\!}
    \begin{tikzcd}[font=\small,column sep=1em]
    \rar & \HGc^{n-*}(A,B;\kktilde) \dar \rar & \HGc^{n-*}(A;\kktilde) \dar \rar & \HGc^{n-*}(B;\kktilde) \dar \rar{\delta} & \HGc^{n+1-*}(A,B;\kktilde) \dar \rar & {} \\
    \rar & \hHG_{*}(X\setminus B,X\setminus A) \rar & \hHG_{*}(X,X\setminus A) \rar & \hHG_{*}(X,X\setminus B) \rar{\delta} & \hHG_{*-1}(X\setminus B,X\setminus A) \rar & {}
  \end{tikzcd}
  \end{equation*}
  all of whose vertical arrows are isomorphisms of \(\RT\)-modules of degree~\(-n\).
  An analogous diagram exists with the roles of homology and cohomology interchanged
  and all arrows reversed.
\end{theorem}

\begin{proposition}
  \label{thm:PAL-duality}
  Assume that \(X\) be an \(n\)-dimensional \(\kk\)-homology manifold
  with equivariant orientation~\(o_{G}\in\hHGc_{n}(X;\kktilde)\).
  Let \((X_{i})\) be a finite \(G\)-stable closed increasing filtration of~\(X\),
  and write \(\hatX_{i}=X\setminus X_{i}\) be the complementary \(G\)-stable open decreasing filtration.
  Then taking the cap product with~\(o_{G}\)
  induces an isomorphism of \(\RT\)-modules (of degree~\(-n\)) from the \(E_{1}\)~page on between the spectral sequences
  \begin{align*}
    E_{1}^{i} &= \HGc^{*}(X_{i},X_{i-1};\kktilde) \;\Rightarrow\;  \HGc^{*}(X;\kktilde), \\
    E_{1}^{i} &= \hHG_{*}(\hatX_{i-1},\hatX_{i}) \;\Rightarrow\;  \hHG_{*}(X).
  \end{align*}
  Similarly, the spectral sequences
  \begin{align*}
    E_{1}^{i} &= \hHGc_{*}(X_{i},X_{i-1};\kktilde) \;\Rightarrow\;  \hHGc_{*}(X;\kktilde), \\
    E_{1}^{i} &= \HG^{*}(\hatX_{i-1},\hatX_{i}) \;\Rightarrow\;  \HG^{*}(X)
  \end{align*}
  are isomorphic from the \(E_{1}\)~page on.
\end{proposition}

\section{Main results}
\label{sec:main}

\subsection{Partial exactness of the Atiyah--Bredon sequence}
\label{sec:partial}

We spell out results for cohomology with closed supports only. They also hold for compact supports.

Let \(X\) be a \(T\)-space. The sequence
\begin{multline}
  \label{eq:barAB}
  \let\longrightarrow\to
  0 \longrightarrow{}
  \HT^{*}(X\AA)  \longrightarrow
  \HT^{*}(\Xp{0}\AA) \longrightarrow
  \HT^{*+1}(\Xp{1},\Xp{0}\AA) \longrightarrow
  \HT^{*+2}(\Xp{2},\Xp{1}\AA) \longrightarrow \\
  \let\longrightarrow\to
  \cdots \longrightarrow
  \HT^{*+r-1}(\Xp{r-1},\Xp{r-2}\AA) \longrightarrow
  \HT^{*+r}(\Xp{r},\Xp{r-1}\AA) \longrightarrow
  0
\end{multline}
is called the augmented \newterm{Atiyah--Bredon sequence}~\(\barABT^{*}(X\AA)\) for~\(X\);
the leftmost term~\(\HT^{*}(X\AA)\) is at position~\(-1\). The first map above is the restriction to~\(\Xp{0}\),
the others are the connecting homomorphisms for the triples~\((X_{i+1},X_{i},X_{i-1})\).

If the term~\(\HT^{*}(X\AA)\) is dropped, we obtain the non-augmented Atiyah--Bredon sequence~\(\ABT^{*}(X\AA)\).
It is the first page of the spectral sequence induced by the orbit filtration
and converging to~\(\HT^{*}(X\AA)\). The second page of this spectral sequence therefore is the cohomology~\(H^{*}(\ABT^{*}(X\AA))\)
of the Atiyah--Bredon sequence.

\begin{lemma}
  \label{thm:barAB-zero}
  If the augmented Atiyah--Bredon sequence~\(\barABT^{*}(X\AA)\)
  is exact at all but possibly two adjacent positions,
  then it is exact everywhere.
  The same holds for the localization of~\(\barABT^{*}(X\AA)\) with respect to any multiplicative subset~\(S\subset R\).
\end{lemma}

\begin{proof}
  This is identical to~\citeorbitsone{Lemma~5.6}. It holds in fact
  for any \(T\)-stable filtration of~\(X\) and any coefficients.
\end{proof}

Let \(\pp\lhd \RT\) be  a prime ideal, and let \(\qq\subset\pp\) be the prime ideal of~\(\RT\) generated by
\begin{equation}
  V=\pp\cap H^{2}(BT;\FF_{p}).
\end{equation}

\begin{lemma}
  \label{thm:Ext-pp-qq}
  Using the above notation, we have for any~\(i\ge0\) the equivalence
  \begin{equation*}
    \Ext_{\RT_{\pp}}^{i}(\HT^{*}(X\AA)_{\pp},\RT_{\pp}) = 0
    \quad\Longleftrightarrow\quad
    \Ext_{\RT_{\qq}}^{i}(\HT^{*}(X\AA)_{\qq},\RT_{\qq}) = 0.
  \end{equation*}
\end{lemma}

\begin{proof}
  The ideal~\(\qq\) is the kernel of the restriction map~\(H^{*}(BT)\to H^{*}(BK)\) for a (non-unique) subtorus \(K\subset T\) with quotient~\(L\).
  Our claim follows from the equivalence
  \begin{equation}
    \Ext_{\RT_{\pp}}^{i}(\HT^{*}(X\AA)_{\pp},\RT_{\pp}) = 0
    \quad\Longleftrightarrow\quad
    \Ext_{H^{*}(BL)}^{i}(H_{L}^{*}(X^{K}\AA),H^{*}(BL)) = 0,
  \end{equation}
  which has been established in~\cite[Lemma~3.2]{Franz:orbits3} for coefficients in a field of characteristic~\(0\).\footnote{%
    The reference for the localization theorem in the proof
    of that result should be~\cite[Ex.~3.1.5, Thm.~3.2.6]{AlldayPuppe:1993}.}
  The proof in characteristic~\(p\) is identical.
\end{proof}

\begin{theorem}
  \label{thm:partial}
  The following are equivalent for any~\(0\le j\le r\).
  \begin{enumarabic}
  \item
    \label{thm:partial-1}
    The augmented Atiyah--Bredon sequence~\(\barABT^{*}(X\AA)\) is exact all positions~\(i\le j-2\).
  \item 
    \label{thm:partial-2}
    \(\HT^{*}(X\AA)\) is a \(j\)-th syzygy over~\(\RT\).
  \item
    \label{thm:partial-3}
    Any linearly independent sequence in~\(H^{2}(BT;\FF_{p})\) of length at most~\(j\)
    is regular on~\(\HT^{*}(X\AA)\).
  \item
    \label{thm:partial-3bis}
    \(\HT^{*}(X\AA)\) is free over~\(H^{*}(BL)\) for any subtorus~\(K\subset T\) of corank at most~\(j\) with quotient~\(L=T/K\).
  \item
    \label{thm:partial-4}
    The restriction map~\(\HT^{*}(X\AA)\to\HK^{*}(X\AA)\) is surjective for all subtori~\(K\subset T\)
    of corank at most~\(j\).
  \end{enumarabic}
\end{theorem}

\begin{proof}
  The equivalence of~\ref{thm:partial-1} and~\ref{thm:partial-2} follows from \Cref{thm:partial-alg}.
  The required assumptions have been established in \Cref{thm:Xi-CM-T} 
  and \Cref{thm:barAB-zero}.
  The implication~\impl{thm:partial-2}{thm:partial-3}
  is valid for any finitely generated \(\RT\)-module~\(M\), see \Cref{thm:char-syzygy}.

  \impl{thm:partial-3}{thm:partial-2}:
  Since the ring~\(\RT\) itself is Cohen--Macaulay,
  we have to prove that
  \begin{equation}
    \label{eq:char-syz-depth}
    \depth_{\RT_{\pp}} \HT^{*}(X\AA)_{\pp} \ge \min(j,\dim \RT_{\pp})
  \end{equation}
  is satisfied for all prime ideals~\(\pp\lhd\RT\), see \Cref{thm:char-syzygy}.

  Let \(\qq\) be as above.
  Because \(\RT_{\qq}\) is a localization of~\(\RT_{\pp}\) (which proves the ``\(\Rightarrow\)''~direction in \Cref{thm:Ext-pp-qq}), we have
  \( 
    \dim \RT_{\qq} \le \dim \RT_{\pp}
  \). 
  Using the characterization~\eqref{eq:depth-Ext} of depth via \(\Ext\),
  we conclude from \Cref{thm:Ext-pp-qq} that it is enough to prove the bound~\eqref{eq:char-syz-depth} for~\(\pp=\qq\).
  
  Recall from~\eqref{eq:local-dim-height} that the dimension of~\(\RT_{\qq}\) equals the vector space dimension of~\(V\) over~\(\FF_{p}\).
  Pick \(k=\min(j,\dim V)\) linearly independent elements from~\(V\). By assumption, they form a regular sequence for~\(\HT^{*}(X)\),
  hence also for its localization at~\(\qq\supset V\). This implies
  \( 
    \depth \HT^{*}(X)_{\qq} \ge k
  \), 
  as was to be shown.

  \impl{thm:partial-2}{thm:partial-3bis}:
  Since \(\HT^{*}(X\AA)\) is a \(j\)-th syzygy over~\(\RT\), it is so over~\(H^{*}(BL)\), where \(L=T/K\).
  Thus \(\HT^{*}(X\AA)\) is free over~\(H^{*}(BL)\) by Hilbert's syzygy theorem (for not necessarily finitely generated graded modules).

  \Equiv{thm:partial-3bis}{thm:partial-4}:
  We assume for the moment that we work with constant coefficients and closed supports
  and consider the bundle
  \begin{equation}
    X_{K} \hookrightarrow X_{T} = (X_{K})_{L} \to BL.
  \end{equation}
  \Cref{thm:nhz} tells us that \(\HT^{*}(X)\) is free over~\(H^{*}(BL)\) if and only if the restriction map
  \begin{equation}
    \HT^{*}(X) = \HL^{*}(X_{K}) \to H^{*}(X_{K}) = \HK^{*}(X)
  \end{equation}
  is surjective.
  For compact supports and/or twisted coefficients, we argue analogously, but on the level of
  the singular Cartan model
  \begin{equation}
    \CT^{*}(X\AA) = \CK^{*}(X\AA) \otimes H^{*}(BL)
  \end{equation}
  where we lift \(L\) to a complement to~\(K\) in~\(T\).
  Since we still have a cohomological ``Serre'' spectral sequence (even for spaces with unbounded cohomology,
  compare~\citeorbitsone{eq.~(3.5)}), the reasoning in \Cref{thm:nhz} carries over.
  
  \impl{thm:partial-3bis}{thm:partial-3}:
  Let \(a_{1}\),~\dots,~\(a_{s}\in H^{2}(BT;\FF_{p})\) be linearly independent for some~\(s\le j\).
  Recall that we can multiply each element by a non-zero scalar and add linear combinations of preceding elements
  without affecting regularity. We may therefore assume that there is a lifting~\(b_{1}\),~\dots,~\(b_{s}\)
  of the sequence to a part of a basis for~\(H^{2}(BT;\Z)\).

  Dividing \(\RT\) by the ideal generated by~\(a_{1}\),~\dots,~\(a_{s}\) therefore corresponds to
  the restriction map~\(\RT\to H^{*}(BK)\) for some subtorus~\(K\subset T\) of corank~\(s\le j\).
  By assumption,
  \(\HT^{*}(X\AA)\) is free over~\(H^{*}(BL)\) where again \(L=T/K\).
  The sequence~\(a_{1}\),~\dots,~\(a_{s}\) consists of images of generators under the inclusion map~\(H^{*}(BL)\to\RT\),
  hence is regular on~\(\HT^{*}(X\AA)\).
\end{proof}

Condition~\ref{thm:partial-3} in \Cref{thm:partial} can be weakened.
Call a subset~\(S\subset H^{2}(BT;\F_{p})\) \newterm{\(j\)-localizing} for~\(X\) and some~\(j\ge0\)
if for any~\(x\in X\) at least \(\min(j,r-\rank_{p} T_{x})\) linearly independent elements from~\(S\)
lie in the kernel of the restriction map~\(\RT=H^{*}(BT)\to H^{*}(BT_{x})\),
where \(\rank_{p}T_{x}\) denotes the rank of the maximal \(p\)-torus contained in~\(T_{x}\).

\begin{proposition}
  \label{thm:localizing}
  Let \(j\ge0\), and let \(S\) be a \(j\)-localizing subset for~\(X\).
  Then \(\HT(X\AA)\) is a \(j\)-th syzygy over~\(\RT\) if and only if every linearly independent sequence in~\(S\)
  of length at most~\(j\) is regular on~\(\HT(X\AA)\).
\end{proposition}

\begin{proof}
  The argument given by Franz--Huang~\cite[Sec.~2]{FranzHuang} for rational \(T\)-equivariant cohomology carries over.
\end{proof}

For a \(G\)-space~\(X\) we consider an augmented Atiyah--Bredon sequence~\(\ABG^{*}(X)\)
analogous to~\eqref{eq:barAB}. \Cref{thm:partial} translates as follows.

\begin{theorem}
  \label{thm:G:partial}
  The following are equivalent for a \(G\)-space~\(X\) and any~\(0\le j\le r\).
  \begin{enumarabic}
  \item
    \label{thm:G:partial-1}
    The augmented Atiyah--Bredon sequence~\(\barABG^{*}(X\AA)\) is exact all positions~\(i\le j-2\).
  \item 
    \label{thm:G:partial-2}
    \(\HG^{*}(X\AA)\) is a \(j\)-th syzygy over~\(\RT=H^{*}(BT)\subset\RG\).
  \item
    \label{thm:G:partial-3}
    Any linearly independent sequence in~\(H^{2}(BT;\FF_{p})\subset\RG\) of length at most~\(j\)
    is regular on~\(\HG^{*}(X\AA)\).
  \end{enumarabic}
\end{theorem}

\begin{proof}
  Everything follows from interpreting the \(G\)-equivariant cohomology of~\(X\)
  as the \(T\)-equivariant cohomology of~\(\indT{X}\) as described in \Cref{thm:ind-HG-HT}.
\end{proof}

We will see in \Cref{sec:remarks-2-odd} that
conditions~\ref{thm:partial-3bis} and~\ref{thm:partial-4} of \Cref{thm:partial} carry over for~\(p=2\), but not for odd~\(p\).
The problem is that the restriction map~\(H^{*}(\indT{X}\AA)\to H^{*}(X\AA)\) may fail to be surjective.

\begin{remark}
  \label{rem:generality}
  As pointed out in the introduction already, our proof of \Cref{thm:partial} applies without any change
  to all previously established cases, that is, to actions of tori~\citeorbitsone{Thm.~5.7},~\citeorbitsfour{Thm.~4.8}
  as well as of compact connected Lie groups~\cite[Thm.~6.10]{Franz:nonab},
  both with coefficients in a field of characteristic~\(0\).
  The crucial Cohen--Macaulay property of the orbit filtration is established
  in~\citeorbitsone{Prop.~4.2},~\citeorbitsfour{Prop.~4.1} and~\cite[Prop.~6.9]{Franz:nonab}.

  The equivalence~\Equiv{thm:partial-1}{thm:partial-2} of \Cref{thm:partial} holds in fact in even greater generality.  
  Let \(E^{*}(-)\) be a cohomology theory of \(R\)-modules where
  \(R\) is a regular ring of dimension~\(r<\infty\).
  Consider a filtration of a space~\(X\) as in~\eqref{eq:def-orbit-filt}
  such that \(E^{*}(X_{i},X_{i-1})\) is Cohen--Macaulay of projective dimension~\(i\) over~\(R\) for each~\(0\le i\le r\).
  Because \Cref{thm:barAB-zero} carries over to~\(E^{*}\), the equivalence of the analogues
  of conditions~\ref{thm:partial-1} and~\ref{thm:partial-2} for~\(E^{*}\) follows again from \Cref{thm:partial-alg}.
\end{remark}

\subsection{Results involving equivariant homology}

We formulate our results only for \(p\)-tori and the usual pair of supports.
Everything remains valid for torus actions and/or the other pair of supports.
However, because of our reduction~\eqref{eq:def-hHG} from \(G\)-equivariant to \(T\)-equivariant homology, sometimes a degree shift is necessary.
We accommodate for this by introducing the \(\RT\)-module~\(\DM\).
We set \(\DM=\RT\) in the torus case and \(\DM=\RT[r]\) for \(p\)-tori.

\begin{proposition}
  \label{thm:hHT-E1}
  The spectral sequence associated to the orbit filtration~\((\Xp{i})\) of~\(X\)
  and converging to~\(\hHG_{*}(X\AA)\) degenerates at the page
  \( 
    E^{1}_{i} = \hHG_{*}(\Xp{i},\Xp{i-1}\AA)
  \). 
\end{proposition}

\begin{proof}
  This follows from the Cohen--Macaulay property of the orbit filtration
  in equivariant homology (\Cref{thm:hHT-hHG-CM}) as in~\citeorbitsone{Cor.~4.4}.
\end{proof}

In \Cref{sec:pd} we defined the complement of the orbit filtration as
the decreasing filtration of~\(X\) by the open \(G\)-stable subsets~\(\hatXp{i}=X\setminus \Xp{i}\) for~\(-1\le i\le r\).

\begin{corollary}
  \label{thm:duflot-topological}
  Assume that \(X\) is a \(\kk \)-homology manifold. Then the spectral sequence associated
  to the filtration~\((\hatXp{i})\) and converging to~\(\HG^{*}(X)\)
  degenerates at the \(E_{1}\)~page.
\end{corollary}

This generalizes a result of Duflot's~\cite[Thm.~1]{Duflot:1983} from smooth manifolds
to homology manifolds. Like Totaro~\cite[Thm.~4.6]{Totaro:2014}, we include the case~\(p=2\) left aside by Duflot.

\begin{proof}
  One combines \Cref{thm:hHT-E1} with equivariant Poincaré--Alexander--Lef\-schetz duality
  as in~\citeorbitsfour{Prop.~4.12}.
\end{proof}

Recall from~\eqref{eq:pairing-hHG-HG} (or~\eqref{eq:pairing-hHT-HT} in the case of tori)
that for any coefficients~\(\ell\) we have a pairing between equivariant homology and cohomology,
\begin{equation}
  \label{eq:pairing-T-and-G}
  \HG^{*}(X\AA) \otimes \hHG_{*}(X\AA) \to \DM.
\end{equation}

\begin{theorem}
  \label{thm:Ext-HAB}
  The following two spectral sequences converging to~\(\HG^{*}(X\AA)\) are naturally isomorphic from the \(E_{2}\)~page on:
  \begin{enumarabic}
  \item the one induced by the orbit filtration with~\(E_{1}^{i}=\HG^{*}(\Xp{i},\Xp{i-1}\AA)\),
  \item the universal coefficient spectral sequence with~\(E_{2}^{i}=\Ext_{\RT}^{i}(\hHG_{*}(X\AA),\DM)\).
  \end{enumarabic}
  Under this isomorphism for the \(E_{2}\)~terms, the restriction map
  \begin{align*}
    \HG^{*}(X\AA) &\to H^{0}(\ABG^{*}(X\AA))\subset\HG^{*}(\Xp{0}\AA) \\
  \intertext{corresponds to the canonical map}
    \HG^{*}(X\AA) &\to \Hom_{\RT}(\hHG_{*}(X\AA),\DM)
  \end{align*}
  induced by the pairing~\eqref{eq:pairing-T-and-G}.
\end{theorem}

\begin{proof}
  This is analogous to \citeorbitsone{Thms.~4.8~\&~5.1} and~\citeorbitsfour{Thm.~4.6}.
  For the comparison of the \(E_{2}\)~pages see also the short proof in~\cite[Sec.~5]{AlldayFranzPuppe:orbits4}.
\end{proof}

\begin{corollary}
  Assume that \(X\) is a \(\kk \)-homology manifold.
  Then the following spectral sequences
  are isomorphic from the \(E_{2}~\)page on:
  \begin{align*}
    E_{1}^{i} &= \hHG_{*}(\hatXp{i-1},\hatXp{i}) \;\Rightarrow\; \hHG_{*}(X), \\
    E_{2}^{i} &= \Ext_{\RT }^{i}(\HG^{*}(X),\DM)  \;\Rightarrow\; \hHG_{*}(X).
  \end{align*}
\end{corollary}

\begin{proof}
  This follows from \Cref{thm:Ext-HAB} together with equivariant Poincaré duality (\Cref{thm:PD-noncompact})
  and equivariant Poincaré--Alexander--Lef\-schetz duality (\Cref{thm:PAL-duality}),
  see~\citeorbitsfour{Cor.~4.9}.
\end{proof}

\subsection{Syzygies and Poincaré duality}
\label{sec:syz-pd}

In this section we only consider equivariant cohomology with closed supports and constant coefficients.
As before, we only spell results out for \(p\)-tori; everything is equally valid for torus actions.

Let \(X\) be a Poincaré duality space, say of dimension~\(n\), with equivariant orientation~\(o_{G}\in\hHG_{n}(X)\).
We then have the equivariant Poincaré pairing
\begin{equation}
  \label{eq:PD-pairing}
  \HG^{*}(X)\times\HG^{*}(X)\to \RT,
  \quad
  (\alpha,\beta) \mapsto \pair{\alpha\cup\beta,o_{G}}
\end{equation}
of degree~\(-n\).

The proofs of the following results are analogous to those given in~\citeorbitsone{Sec.~5.3}.

\begin{proposition}
 The equivariant Poincaré pairing~\eqref{eq:PD-pairing} is non-degenerate
 if and only if the \(\RT\)-module \(\HG^{*}(X)\) is torsion-free.
\end{proposition}

\goodbreak

\begin{proposition}
  The following conditions are equivalent:
  \begin{enumarabic}
  \item \label{G:p1} The Chang--Skjelbred sequence~\eqref{eq:intro:cs} is exact.
  \item \label{G:p3} The \(\RT\)-module~\(\HG^{*}(X)\) is reflexive.
  \item \label{G:p2} The equivariant Poincaré pairing~\eqref{eq:PD-pairing} is perfect.
  \end{enumarabic}
\end{proposition}

\begin{proposition}
  \label{thm:syz-bound-orient}
  If \(\HG^{*}(X)\) is a syzygy of order~\(\ge r/2\), then it is free over~\(\RT\).
\end{proposition}

\section{Examples}
\label{sec:examples}

\subsection{Non-compact examples}

Assume \(r>0\) and
consider the non-compact manifold
\begin{equation}
  X = (S^{2})^{r}\setminus\bigl\{(N,\dots,N),(S,\dots,S)\bigr\}
\end{equation}
where \(N\) and~\(S\) denote the north and south pole of the \(2\)-sphere, respectively.
The torus~\(T\) acts on~\(X\) by rotating each sphere in the usual way.
In~\citeorbitsone{Sec.~6.1} we computed the rational \(T\)-equivariant cohomology of~\(X\)
and showed that it is a syzygy of order~\(r-1\).

The argument in~\cite{AlldayFranzPuppe:orbits1} works over any field because the description
of the equivariant cohomology of a smooth toric variety used there is valid for any coefficients.
Hence we get a syzygy of order~\(r-1\) in \(T\)-equivariant cohomology with coefficients in~\(\kk\)
and by \Cref{thm:red-HG-HT-alg}\,\ref{thm:red-HG-HT-syz} also in \(G\)-equivariant cohomology.

Letting another torus act trivially on~\(X\) gives by \Cref{thm:syz-prod}
examples of any syzygy order less than~\(r\). One can also show that the equivariant cohomology of the manifold
\begin{equation}
  Y = (S^{2})^{r}\setminus\bigl\{(N,\dots,N),(\underbrace{S,\dots,S}_{m},\underbrace{N,\dots,N}_{r-m})\bigr\}
\end{equation}
with the analogous action of~\(T\) or~\(G\)
is a syzygy of order~\(m-1\) for any~\(1\le m\le r\).
If one removes open balls around the fixed points instead, one obtains compact orientable manifolds with boundary
having the same equivariant cohomology.

\subsection{Big polygon spaces}

In~\cite{Franz:maximal} a family of compact orientable \(T\)-manifolds, the so-called big polygon spaces, were constructed
whose rational \(T\)-equivariant cohomologies realize all syzygy orders less than~\(r\) over~\(H^{*}(BT;\Q)\)
that are permitted by the analogue~\citeorbitsone{Prop.~5.12} of \Cref{thm:syz-bound-orient}.
The exact syzygy order has been determined in~\cite[Thm.~3.2]{FranzHuang} for all big polygon spaces.

The computation of the rational \(T\)-equivariant cohomology in~\cite[Secs.~4--6]{Franz:maximal}
can actually be carried out over any field. Hence these examples remain valid
for \(T\)-equivariant cohomology with coefficients in~\(\kk\). Moreover, the argument
in~\cite[Sec.~3]{FranzHuang} is purely algebraic and again works over any field
(see also \Cref{thm:localizing}).
If we restrict the torus action on big polygon spaces to the \(p\)-torus,
we therefore get examples of compact orientable \(G\)-manifolds whose equivariant cohomologies
realize all syzygy orders less than~\(r\) allowed by~\Cref{thm:syz-bound-orient}.

For~\(p=2\) another family of examples is given by the real versions of big polygon spaces
studied in~\cite{Puppe:2018} and~\cite[Sec.~5]{FranzHuang}, with analogous properties.
Chaves~\cite[Sec.~6]{Chaves:semi} 
has related their syzygy orders to those of (complex) big polygon spaces
by looking at them as fixed point sets of the complex conjugation on the latter
and computing the equivariant cohomology with respect to the semidirect
product of~\(T\) and the group~\(\Z_{2}\) acting via the involution.

\section{Remarks about the cases~\(p=2\) and odd~\(p\)}
\label{sec:remarks-2-odd}

\subsection{The case~\(p=2\)}
\label{sec:remarks-2}

For \(2\)-tori one can consider equivariant cohomology also as a module over the polynomial ring~\(\RG=\kk[x_{1},\dots,x_{r}]\)
instead of~\(\RT=\kk[x_{1}^{2},\dots,x_{r}^{2}]\). We content ourselves with stating analogues
of \Cref{thm:Xi-CM-G} and~\Cref{thm:partial} in this setting. In particular, we establish an analogue for \(2\)-tori
of condition~\ref{thm:partial-4} in \Cref{thm:partial}. It has been left out from \Cref{thm:G:partial}
because it may fail for odd~\(p\), as the example in the following section shows.

\begin{proposition}
  \label{thm:2:Xi-CM-G}
  Let \(X\) be a \(G\)-space. The \(\RG\)-module~\(\HG^{*}(\Xp{i},\Xp{i-1})\)
  is zero or Cohen--Macaulay of projective dimension~\(i\) for any~\(0\le i\le r\).
\end{proposition}

\begin{proof}
  This follows from the fact that depth and dimension over~\(\RG\) and~\(\RT\) coincide,
  compare~\cite[Rem.~A.6.19\,(2)~\&~(3)]{AlldayPuppe:1993}. 
  Alternatively, one can verify that the proof of \Cref{thm:Xi-CM-G} remains valid for \(\RG\) instead of~\(\RT\).
\end{proof}

\begin{theorem}
  \label{thm:2:partial}
  The following are equivalent for any \(G\)-space~\(X\) and any~\(0\le j\le r\).
  \begin{enumarabic}
  \item
    \label{thm:2:partial-1}
    The augmented Atiyah--Bredon sequence~\(\barABG^{*}(X)\) is exact all positions~\(i\le j-2\).
  \item 
    \label{thm:2:partial-2}
    \(\HG^{*}(X)\) is a \(j\)-th syzygy over~\(\RG\).
  \item
    \label{thm:2:partial-3}
    Any linearly independent sequence in~\(H^{1}(BG;\FF_{2})\) of length at most~\(j\)
    is regular on~\(\HG^{*}(X)\).
  \item
    \label{thm:2:partial-3bis}
    \(\HG^{*}(X)\) is free over~\(H^{*}(BL)\) for any sub-\(2\)-torus~\(K\subset T\) of corank at most~\(j\) with quotient~\(L=G/K\).
  \item
    \label{thm:2:partial-4}
    The restriction map~\(\HG^{*}(X\AA)\to\HK^{*}(X\AA)\) is surjective for all sub-\(2\)-tori \(K\subset G\)
    of corank at most~\(j\).
  \end{enumarabic}
\end{theorem}

\begin{proof}
  We compare the conditions~\ref{thm:2:partial-1} to~\ref{thm:2:partial-3}
  above with the corresponding ones in \Cref{thm:G:partial},
  which we label \refstar{thm:partial-1} to~\refstar{thm:partial-3} in this proof.
  The conditions~\ref{thm:2:partial-1} and~\refstar{thm:partial-1} are identical.

  It is clear that \ref{thm:2:partial-2} implies \refstar{thm:partial-2} since the \(\RG\)-free modules~\(F^{i}\) appearing
  in the exact sequence~\eqref{eq:def-syzygy} are also \(\RT\)-free. Conversely, if \(\HG^{*}(X)\) is a \(j\)-th syzygy over~\(\RT\),
  then so is \(\HT^{*}(X)\) by \Cref{thm:red-HG-HT-alg}\,\ref{thm:red-HG-HT-syz}. Tensoring an exact sequence
  of the form~\eqref{eq:def-syzygy} for~\(\HT^{*}(X)\) with~\(\RG\) over~\(\RT\) displays \(\HG^{*}(X)\) as a \(j\)-th syzygy by \Cref{thm:red-HG-HT}.

  A homogeneous sequence~\(a_{1}\),~\dots,~\(a_{k}\in H^{>0}(BG)\) is regular
  on a finitely generated graded \(\RG\)-module~\(M\) if and only if the squared sequence~\(a_{1}^{2}\),~\dots,~\(a_{k}^{2}\) is so,
  \cf~\cite[Prop., p.~1]{Hochster:CM}.
  Hence \ref{thm:2:partial-3} and~\refstar{thm:partial-3} are equivalent.

  The implications~\implthree{thm:2:partial-2}{thm:2:partial-3bis}{thm:2:partial-3}
  and~\Equiv{thm:2:partial-3bis}{thm:2:partial-4}
  are done as in the proof of \Cref{thm:partial}, based on \Cref{thm:nhz}.
\end{proof}

Analogously to \Cref{thm:localizing}, one can relax condition~\ref{thm:2:partial-3}
by using only sequences contained in a \(j\)-localizing subset~\(S\subset H^{1}(BG;\F_{2})\), see \cite[Lemma~5.2]{FranzHuang}.

A different approach is to work over~\(\RG\) from the outset, so that for example \(G\)-equivariant homology
becomes a module over~\(\RG\) and the equivariant Poincaré duality isomorphism \(\RG\)-linear.
This has been carried out by Chaves~\cite[Sec.~3]{Chaves:thesis}, based on an earlier draft of the present paper.

\subsection{The case of odd~\(p\)}

Let \(X\) be a \(G\)-space. For simplicity, we only consider cohomology with closed supports and constant coefficients.

Recall that the augmentation ideal~\(\mm\lhd\kk[G]\) is nilpotent. (More precisely, one has \(\mm^{(p-1)r+1}=0\), also for~\(p=2\).)
This implies that any \(G\)-action on a \(\kk\)-vector space is nilpotent.
\Cref{thm:nhz} therefore implies
that \(\HG^{*}(X)\) is free over \(\RG\) if and only if the restriction map~\(\HG^{*}(X)\to H^{*}(X)\)
is surjective, and these conditions imply that \(G\) acts trivially on~\(H^{*}(X)\).
In contrast to actions of \(2\)-tori discussed above, however,
\(\HG^{*}(X)\) may be free over~\(\RT\) without being so over~\(\RG\).

An example is the suspension~\(X=\Sigma\,G\) of~\(G=\Z_{p}\).
The augmented Atiyah--Bredon sequence for~\(X\) is of the form
\begin{equation}
  \label{eq:AB-susp}
  \let\to\longrightarrow
  0 \to \HG^{*}(X) \to \RG \oplus \RG \to \kk \to 0
\end{equation}
and exact, which gives isomorphisms of \(\RT\)-modules
\begin{equation}
  \HG^{*}(X) \cong \RT \oplus \RT[1] \oplus \RT[1] \oplus \RT[2] \cong \RG \oplus \RG[1].
\end{equation}
A look at~\eqref{eq:AB-susp} shows that
\(\HG^{*}(X)\) and~\(\RG \oplus \RG[1]\) cannot be isomorphic over~\(\RG\)
because \(\HG^{1}(X)\) is annihilated by the generator~\(x\) of the exterior algebra~\(\Lambda\).
That \(\HG^{*}(X)\) is not free over~\(\RG\) also follows from \Cref{thm:nhz} because
\(G\) acts non-trivially on~\(H^{*}(X)\), so that \(\HG^{*}(X)\) cannot surject onto~\(H^{*}(X)\).
On the other hand, \(G\) does act trivially on~\(H^{*}(X)\) for~\(p=2\),
and \(\HG^{*}(X)\) surjects onto it and is free over~\(\RG\).


\begin{thebibliography}{99}

\bibitem{Allday:1975}
C.~Allday,
\newblock Torus actions on a cohomology product of three odd spheres,
\newblock \textit{Trans.\ Amer.\ Math.\ Soc.}~\textbf{203} (1975), 343--358;
\newblock \doi{10.1090/S0002-9947-1975-0377953-3}

\bibitem{AlldayFranzPuppe:orbits1}
C.~Allday, M.~Franz, V.~Puppe,
\newblock Equivariant cohomology, syzygies and orbit structure,
\newblock \textit{Trans.\ Amer.\ Math.\ Soc.}~\textbf{366} (2014), 6567--6589;
\newblock \doi{10.1090/S0002-9947-2014-06165-5}

\bibitem{AlldayFranzPuppe:orbits4}
C.~Allday, M.~Franz, V.~Puppe,
\newblock Equivariant Poincaré--Alexander--Lefschetz duality and the Cohen--Macaulay property,
\newblock \textit{Alg.\ Geom.\ Top.}~\textbf{14} (2014), 1339--1375;
\newblock \doi{10.2140/agt.2014.14.1339}

\bibitem{AlldayPuppe:1993}
C.~Allday, V.~Puppe,
\newblock \textit{Cohomological methods in transformation groups},
\newblock Cambridge Univ.\ Press, Cambridge 1993

\bibitem{Bredon:1972}
G.~E.~Bredon,
\newblock \textit{Introduction to compact transformation groups},
\newblock Academic Press, New York 1972

\bibitem{BrunsVetter:1988}
W.~Bruns, U.~Vetter,
\newblock \textit{Determinantal rings},
\newblock LNM~\textbf{1327}, Springer, Berlin 1988;
\newblock \doi{10.1007/BFb0080378};
\newblock available at \url{http://www.home.uni-osnabrueck.de/wbruns/brunsw/detrings.pdf}

\bibitem{BrueskeIschebeckVogel:1989}
R.~Brüske, F.~Ischebeck, F.~Vogel,
\newblock \textit{Kommutative Algebra},
\newblock BI-Wissenschaftsverlag, Mannheim 1989;
\newblock available at \url{http://wwwmath.uni-muenster.de/u/ischebeck/SkriptBrskeIschebeckVogel.pdf}

\bibitem{ChangSkjelbred:1974}
T.~Chang, T.~Skjelbred,
\newblock The topological Schur lemma and related results,
\newblock \textit{Ann.\ Math.~(2)}~\textbf{100} (1974), 307--321;
\newblock \doi{10.2307/1971074}

\bibitem{Chaves:thesis}
S.~Chaves Ramirez,
\newblock Equivariant cohomology for $2$-torus actions and torus actions with compatible involutions,
\newblock Ph.\,D.\ thesis, Univ.\ Western Ontario 2020;
\newblock available at \url{http://ir.lib.uwo.ca/etd/7049/}

\bibitem{Chaves:semi}
S.~Chaves,
\newblock The equivariant cohomology for semidirect product actions,
\newblock \arxiv{2009.08526v2}

\bibitem{Duflot:1983}
J.~Duflot,
\newblock Smooth toral actions,
\newblock \textit{Topology}~\textbf{22} (1983), 253--265;
\newblock \doi{10.1016/0040-9383(83)90012-5}

\bibitem{Dwyer:1974}
W.~G.~Dwyer, Strong convergence of the Eilenberg--Moore spectral sequence,
\newblock \textit{Topology}~\textbf{13} (1974), 255--265;
\newblock \doi{10.1016/0040-9383(74)90018-4}

\bibitem{Eisenbud:2005}
D.~Eisenbud,
\newblock \textit{The geometry of syzygies},
\newblock GTM~{\bf 229}, Springer, New York 2005;
\newblock \doi{10.1007/b137572}

\bibitem{Franz:maximal}
M.~Franz,
\newblock Big polygon spaces,
\newblock \textit{Int.\ Math.\ Res.\ Not.}~\textbf{2015} (2015), 13379--13405;
\newblock \doi{10.1093/imrn/rnv090}

\bibitem{Franz:nonab}
M.~Franz,
\newblock Syzygies in equivariant cohomology for non-abelian Lie groups,
\newblock pp.~325-360 in: F. Callegaro \emph{et al.} (eds.),
\newblock Configuration spaces (Cortona, 2014),
\newblock \textit{Springer INdAM Ser.}~\textbf{14}, Springer, Cham 2016;
\newblock \doi{10.1007/978-3-319-31580-5\_14}

\bibitem{Franz:orbits3}
M.~Franz,
\newblock A quotient criterion for syzygies in equivariant cohomology,
\newblock \textit{Transformation Groups}~\textbf{22} (2017), 933--965;
\newblock \doi{10.1007/s00031-016-9408-3}

\bibitem{FranzHuang}
M.~Franz, J.~Huang,
\newblock The syzygy order of big polygon spaces,
\newblock \textit{Alg.\ Geom.\ Top.}~\textbf{20} (2020), 2657--2675;
\newblock \doi{10.2140/agt.2020.20.2657}

\bibitem{FranzPuppe:2011}
M.~Franz, V.~Puppe,
\newblock Exact sequences for equivariantly formal spaces,
\newblock \textit{C.\ R.\ Math.\ Acad.\ Sci.\ Soc.\ R.\ Can.}~\textbf{33} (2011), 1--10

\bibitem{GoreskyKottwitzMacPherson:1998}
M.~Goresky, R.~Kottwitz, R.~MacPherson,
\newblock Equivariant cohomology, Koszul duality, and the localization theorem,
\newblock \textit{Invent.\ Math.}~\textbf{131} (1998), 25--83;
\newblock \doi{10.1007/s002220050197}

\bibitem{Hatcher:2002}
A.~Hatcher,
\newblock \textit{Algebraic topology},
\newblock Cambridge Univ.\ Press, Cambridge 2002

\bibitem{Hochster:CM}
M.~Hochster, \textit{Cohen--Macaulay rings},
\newblock unpublished course notes (2014);
\newblock available at \url{http://www.math.lsa.umich.edu/~hochster/615W14/CM.pdf}

\bibitem{Lang:1993}
S.~Lang,
\newblock \textit{Algebra}, 3rd ed.,
\newblock Addison-Wesley, Reading, MA 1993

\bibitem{Murayama:1983}
M.~Murayama,
\newblock On $G$-ANR's and their $G$-homotopy types,
\newblock \textit{Osaka J.\ Math.}~\textbf{20} (1983), 479--512;
\newblock available at \url{http://projecteuclid.org/euclid.ojm/1200776318}

\bibitem{Puppe:2018}
V.~Puppe,
\newblock Equivariant cohomology of $(\mathbb{Z}_{2})^{r}$-manifolds and syzygies,
\newblock \textit{Fund.\ Math.}~\textbf{243} (2018), 55--74;
\newblock \doi{10.4064/fm405-12-2017}

\bibitem{Smith:1967}
L.~Smith,
\newblock Homological algebra and the Eilenberg--Moore spectral sequence,
\newblock \textit{Trans.\ Amer.\ Math.\ Soc.}~\textbf{129} (1967), 58--93;
\newblock \doi{10.2307/1994364}

\bibitem{Totaro:2014}
B.~Totaro,
\newblock Group cohomology and algebraic cycles,
\newblock Cambridge Univ.\ Press, Cambridge 2014;
\newblock \doi{10.1017/CBO9781139059480}

\end{thebibliography}
\end{document}